%% file: preprint.tex
\documentclass[12pt,a4paper]{article}

\usepackage{a4}
\input{header}

\usepackage[margin=3cm]{geometry}

\usepackage[final,
      pdftitle={Residual based Error Estimate and Quasi-Interpolation on Polygonal Meshes for High Order BEM-based FEM},
      pdfauthor={Steffen Weisser},
      pdfkeywords={BEM-based FEM; polygonal finite elements; quasi-interpolation; Poincare constant; residual based error estimate},
        colorlinks, % Schrift in Farbe, sonst mit Rahmen
        bookmarksnumbered, % Inhaltsverzeichnis mit Numerierung
        bookmarksopen, % "offnet das Inhaltsverzeichnis
        pdfstartview=FitH, % startet mit Seitenbreite
        linkcolor=black, % standard red
        citecolor=black, % standard green
        urlcolor=black, % standard cyan
        filecolor=black %
]{hyperref}

\title{Residual based Error Estimate and Quasi-Interpolation on Polygonal Meshes for High Order BEM-based FEM\footnotetext{Submitted to journal: 03.11.2015}}
\author{Steffen Wei\ss er\footnote{Saarland University, Department of Mathematics, 66041 Saarbr\"ucken, Germany}}

\begin{document}

\maketitle

\begin{abstract}
Only a few numerical methods can treat boundary value problems on polygonal and polyhedral meshes. The BEM-based Finite Element Method is one of the new discretization strategies, which make use of and benefits from the flexibility of these general meshes that incorporate hanging nodes naturally. The article in hand addresses quasi-interpolation operators for the approximation space over polygonal meshes. To prove interpolation estimates the Poincar\'e constant is bounded uniformly for patches of star-shaped elements. These results give rise to the residual based error estimate for high order BEM-based FEM and its reliability as well as its efficiency are proven. Such a posteriori error estimates can be used to gauge the approximation quality and to implement adaptive FEM strategies. Numerical experiments show optimal rates of convergence for meshes with non-convex elements on uniformly as well as on adaptively refined meshes.

{\footnotesize
\noindent\textbf{Keywords} BEM-based FEM $\cdot$ polygonal finite elements $\cdot$ quasi-interpolation $\cdot$ Poincar\'e constant $\cdot$ residual based error estimate\\[0.5ex]
\noindent\textbf{Mathematics Subject Classification (2000)} 65N15 $\cdot$ 65N30 $\cdot$ 65N38 $\cdot$ 65N50
}
\end{abstract}

\input{Introduction.tex}
\input{ModelProblemDiscretization.tex}
\input{QuasiInterpolation.tex}
\input{ResidualBasedErrorEstimate.tex}

\input{NumericalExperiments.tex}

\input{Conclusion.tex}

%\bibliographystyle{plain}
%\bibliography{literatur}

\end{document}

%% file: header.tex
\usepackage{latexsym}

\usepackage[intlimits]{amsmath}
\usepackage{amsfonts}
\usepackage{amssymb}
\usepackage{amsthm}

\usepackage{stmaryrd} % for \llbracket and \rrbracket

\usepackage{exscale}

\usepackage{comment}

\theoremstyle{plain}
\newtheorem{theorem}{Theorem}
\newtheorem{proposition}{Proposition}
\newtheorem{lemma}{Lemma}
\theoremstyle{definition}
\newtheorem{definition}{Definition}
\newtheorem{remark}{Remark}

\renewcommand{\div}{\mathrm{div}} % for divergence
\newcommand{\supp}{\mathrm{supp\,}} % for support
\newcommand{\Span}{\mathrm{span\,}} % for span

\usepackage{graphicx}
\DeclareGraphicsExtensions{.pdf, .jpg, .png}
\usepackage{color}
\graphicspath{{fig/}}

%% file: Introduction.tex
\section{Introduction}
\label{sec:Introduction}
Adaptive finite element strategies play a crucial role in nowadays efficient implementations for the numerical solution of boundary value problems. Due to a posteriori error control, the meshes are only refined locally to reduce the computational cost and to increase the accuracy. For classical discretization techniques, as the Finite Element Method (FEM), triangular or quadrilateral (2D) and tetrahedral or hexahedral (3D) elements are applied and one has to take care to preserve the admissibility of the mesh after local refinements~\cite{Verfuerth2013}. Polygonal (2D) and polyhedral (3D) meshes are, therefore, very attractive for such local refinements. They allow a variety of element shapes and they naturally treat hanging nodes in the discretization, which appear as common nodes on a more general element. But only a few methods can handle such meshes which also provide a more flexible and direct way to meshing for complex geometries and interfaces. 

The BEM-based Finite Element Method is one of the new promising approaches applicable on general polygonal and polyhedral meshes. It was first introduced in~\cite{CopelandLangerPusch2009} and analysed in~\cite{HofreitherLangerPechstein2010,Hofreither2011}. The method makes use of local Trefftz-like trial functions, which are defined implicitly as local solutions of boundary value problems. These problems are treated by means of Boundary Element Methods~(BEM) that gave the name. The BEM-based FEM has been generalized to high order approximations~\cite{RjasanowWeisser2012,Weisser2014,Weisser2014PAMM}, mixed formulations with $H(\mathrm{div})$ conforming approximations~\cite{EfendievGalvisLazarovWeisser2014} as well as to convection-adapted trial functions~\cite{HofreitherLangerWeisser2015}. Furthermore, the strategy has been applied to general polyhedral meshes~\cite{RjasanowWeisser2014} and time dependent problems~\cite{Weisser2014WCCM}. For efficient computations fast FETI-type solvers have been developed for solving the resulting large linear systems of equations~\cite{HofreitherLangerPechstein2014}. Additionally, the BEM-based FEM has shown its flexibility and applicability on adaptively refined polygonal meshes~\cite{Weisser2011,Weisser2015Enumath}. 

In the last few years, polygonal and polyhedral meshes attracted a lot of interest. New methods have been developed and conventional approaches were mathematically revised to handle them. The most prominent representatives for the new approaches beside of the BEM-based FEM are the Virtual Element Method~\cite{BeiraoDaVeigaBrezziCangianiManziniMariniRusso2012} and the Weak Galerkin Method~\cite{WangYe2014}. Strategies like discontinuous Galerkin~\cite{DolejsiFeistauerSobotikova2005} and the mimetic discretization techniques~\cite{BeiraoDaVeigaLipnikovManzini2011} are also considered on polygonal and polyhedral meshes. But there are only a few references to adaptive strategies and a posteriori error control. A posteriori error estimates for the discontinuous Galerkin method are given in~\cite{KarakashianPascal2003}. To the best of our knowledge there is only one publication for the Virtual Element Method~\cite{BeiraoDaVeigaManzini2015} and one for the Weak Galerkin Method~\cite{ChenWangYe2014}. The first mentioned publication deals with a residual a posteriori error estimate for a $C^1$-conforming approximation space, and the second one is limited to simplicial meshes. For the mimetic discretization technique there are also only few references which are limited to low order methods, see the recent work~\cite{AntoniettiBeiraoDaVeigaLovadinaVerani2013}. 

The article in hand contributes to this short list of literature. In Section~\ref{sec:ModelProblemDiscretization}, some notation and the model problem are given. The regularity of polygonal meshes is discussed and useful properties are proven. Additionally, the BEM-based FEM with high order approximation spaces is reviewed. Section~\ref{sec:QuasiInterpolation} deals with quasi-interpolation operators of Cl\'ement type for which interpolation estimates are shown on regular polygonal meshes with non-convex elements. These operators are used in Section~\ref{sec:ResidualErrorEstimate} to derive a residual based error estimate for the BEM-based FEM on general meshes which is reliable and efficient. The a posteriori error estimate can be applied to adaptive mesh refinement that yields optimal orders of convergence in the numerical examples in Section~\ref{sec:NumericalExperiments}.

%% file: ModelProblemDiscretization.tex
\section{Model problem and discretization}
\label{sec:ModelProblemDiscretization}
Let $\Omega\subset\mathbb R^2$ be a connected, bounded, polygonal domain with boundary $\Gamma=\Gamma_D\cup\Gamma_N$, where $\Gamma_D\cap\Gamma_N = \emptyset$ and $|\Gamma_D|>0$. We consider the diffusion equation with mixed Dirichlet--Neumann boundary conditions
\begin{equation}\label{eq:BVP}
   -\mathrm{div}(a\nabla u) = f\quad\mbox{in }\Omega, \qquad 
   u=g_D\quad\mbox{on }\Gamma_D, \qquad 
   %a\nabla u\cdot n_\Omega=g_N\quad\mbox{on }\Gamma_N,
   a\frac{\partial u}{\partial n_\Omega}=g_N\quad\mbox{on }\Gamma_N,
\end{equation}
where $n_\Omega$ denotes the outer unite normal vector to the boundary of~$\Omega$. For simplicity, we restrict ourselves to piecewise constant and scalar valued diffusion $a\in L_\infty(\Omega)$ with $0<a_\mathrm{min}\leq a\leq a_\mathrm{max}$ almost everywhere in $\Omega$. Furthermore, we assume that the discontinuities of~$a$ are aligned with the initial mesh in the discretization later on. The treatment of continuously varying coefficients is discussed in~\cite{RjasanowWeisser2012}, where the coefficient is approximated by a piecewise constant function or by a proper interpolation in a BEM-based approximation space of lower order.

The usual notation for Sobolev spaces is utilized, see~\cite{Adams1975,McLean2000}. Thus, for an open subset $\omega\subset\Omega$, which is either a two dimensional domain or a one dimensional Lipschitz manifold, the Sobolev space $H^s(\omega)$, $s\in\mathbb R$ is equipped with the norm $\|\cdot\|_{s,\omega}=\|\cdot\|_{H^s(\omega)}$ and semi-norm $|\cdot|_{s,\omega}=|\cdot|_{H^s(\omega)}$. This notation includes the Lebesgue space of square integrable functions for $s=0$, since $H^0(\omega)=L_2(\omega)$. The $L_2(\omega)$-inner product is abbreviated to $(\cdot,\cdot)_\omega$.

For $g_D\in H^{1/2}(\Gamma_D)$, $g_N\in L_2(\Gamma_N)$ and $f\in L_2(\Omega)$, the well known weak formulation of problem~\eqref{eq:BVP} reads
\begin{equation}\label{eq:VF}
 \mbox{Find } u\in g_D+V:\quad b(u,v) = (f,v)_\Omega+(g_N,v)_{\Gamma_N}\quad \forall v\in V,
\end{equation}
and admits a unique solution, since the bilinear form
\[b(u,v) = (a\nabla u,\nabla v)_\Omega\]
is bounded and coercive on
\[V=\{v\in H^1(\Omega):v=0\mbox{ on } \Gamma_D\}.\]
%\[b(u,v) = \int_\Omega a\nabla u\cdot\nabla v \qquad\mbox{and}\qquad V=\{v\in H^1(\Omega):v=0\mbox{ on } \Gamma_D\}.\]
Here, $g_D+V\subset H^1(\Omega)$ denotes the affine space which incorporates the Dirichlet datum. Note that the same symbol $g_D$ is used for the Dirichlet datum and its extension into $H^1(\Omega)$. 

\begin{figure}[tbp]
 \centering
 \scalebox{0.8}{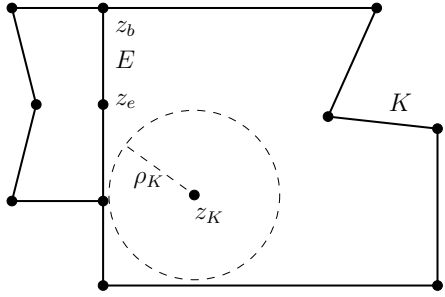}
 \vspace*{-1ex}
 \caption{Two neighbouring elements in a polygonal mesh, nodes are marked with dots}
 \label{fig:regMesh}
\end{figure}
The domain~$\Omega$ is decomposed into a family of polygonal discretizations $\mathcal K_h$ consisting of non-overlapping, star-shaped elements $K\in\mathcal K_h$ such that
%\[\overline\Omega = \bigcup\{x\in\overline K:K\in\mathcal K_h\}.\]
\[\overline\Omega = \bigcup_{K\in\mathcal K_h}\overline K.\]
Their boundaries~$\partial K$ consist of nodes and edges, cf.\ Figure~\ref{fig:regMesh}. 
An edge $E=\overline{z_bz_e}$ is always located between two nodes, the one at the beginning~$z_b$ and the one at the end~$z_e$. We set $\mathcal N(E)=\{z_b,z_e\}$, these points are fixed once per edge and they are the only nodes on $E$. 
The sets of all nodes and edges in the mesh are denoted by $\mathcal N_h$ and $\mathcal E_h$, respectively. The sets of nodes and edges corresponding to an element $K\in\mathcal K_h$ are abbreviated to $\mathcal N(K)$ and $\mathcal E(K)$. Later on, the sets of nodes and edges which are in the interior of $\Omega$, on the Dirichlet boundary~$\Gamma_D$ and on the Neumann boundary~$\Gamma_N$ are needed. Therefore, we decompose $\mathcal N_h$ and $\mathcal E_h$ into $\mathcal N_h = \mathcal N_{h,\Omega} \cup \mathcal N_{h,D} \cup \mathcal N_{h,N}$ and $\mathcal E_h = \mathcal E_{h,\Omega} \cup \mathcal E_{h,D} \cup \mathcal E_{h,N}$.
The length of an edge~$E$ and the diameter of an element~$K$ are denoted by $h_E$ and $h_K=\sup\{\vert x-y\vert:x,y\in\partial K\}$, respectively. 

\subsection{Polygonal mesh}
In order to proof convergence and approximation estimates the meshes have to satisfy some regularity assumptions. We recall the definition gathered from~\cite{Weisser2014} which is needed in the remainder of the presentation.
\begin{definition}\label{def:regularMesh}
 The family of meshes $\mathcal K_h$ is called regular if it fulfills:
 \begin{enumerate}
  \item Each element $K\in\mathcal K_h$ is a star-shaped polygon with respect to a circle of radius~$\rho_K$ and midpoint~$z_K$.
  \item The aspect ratio is uniformly bounded from above by $\sigma_\mathcal{K}$, i.e. \\$h_K/\rho_K<\sigma_\mathcal{K}$ for all $K\in\mathcal K_h$.
  \item There is a constant $c_\mathcal{K}>0$ such that for all elements $K\in\mathcal K_h$ and all its edges $E\subset\partial K$ it holds $h_K\leq c_\mathcal{K}h_E$.
 \end{enumerate}
\end{definition}
The circle in the definition is chosen in such a way that its radius is maximal, cf.\ Figure~\ref{fig:regMesh}. If the position of the circle is not unique, its midpoint~$z_K$ is fixed once per element. Without loss of generality, we assume $h_K<1$, $K\in\mathcal K_h$ that can always be satisfied by scaling of the domain. 

\begin{figure}[tbp]
 \centering
 \scalebox{0.8}{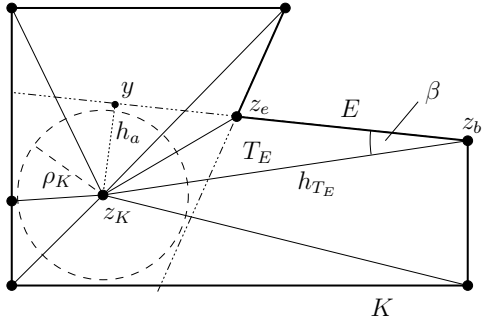}
 \vspace*{-1ex}
 \caption{Auxiliary triangulation $\mathcal T_h(K)$ of star-shaped element $K$, altitude $h_a$ of triangle $T_E\in\mathcal T_h(K)$ perpendicular to $E$ and angle $\beta$}
 \label{fig:regMesh_auxTria}
\end{figure}
In the following, we give some useful properties of regular meshes. An important analytical tool is an auxiliary triangulation $\mathcal T_h(K)$ of the elements $K\in\mathcal K_h$. This triangulation is constructed by connecting the nodes on the boundary of $K$ with the point $z_K$ of Definition~\ref{def:regularMesh}, see Figure~\ref{fig:regMesh_auxTria}. Consequently, $\mathcal T_h(K)$ consists of the triangles $T_E$ for $E=\overline{z_bz_e}\in\mathcal E(K)$, which are defined by the points $z_b$, $z_e$ and $z_K$.

\begin{lemma}\label{lem:regAuxTria}
Let $K$ be a polygonal element of a regular mesh $\mathcal K_h$. The auxiliary triangulation $\mathcal T_h(K)$ is shape-regular in the sense of Ciarlet~\cite{Ciarlet1978}, i.e. neighbouring triangles share either a common node or edge and the aspect ratio of each triangle is uniformly bounded by some constant~$\sigma_{\mathcal T}$, which only depends on $\sigma_\mathcal{K}$ and $c_\mathcal{K}$.
\end{lemma}
\begin{proof}
Let $T_E\in\mathcal T_h(K)$ be a triangle with diameter~$h_{T_E}$ and let $\rho_{T_E}$ be the radius of the incircle. It is known that the area of~$T_E$ is given by $|T_E|=\frac{1}{2}|\partial T_E|\rho_{T_E}$, where $|\partial T_E|$ is the perimeter of~$T_E$. Obviously, it is $|\partial T_E|\leq3h_{T_E}$. On the other hand, we have the formula $|T_E|=\frac{1}{2}h_Eh_a$, where $h_a$ is the altitude of the triangle perpendicular to~$E$, see Figure~\ref{fig:regMesh_auxTria}. Since the element~$K$ is star-shaped with respect to a circle of radius~$\rho_K$, the line through the side $E\in\mathcal E_h$ of the triangle does not intersect this circle. Thus, $h_a\geq\rho_K$ and we have the estimate $|T_E|\geq \frac{1}{2}h_E\rho_K$. Together with Definitions~\ref{def:regularMesh}, we obtain
\[
 \frac{h_{T_E}}{\rho_{T_E}} 
 = \frac{|\partial T_E|h_{T_E}}{2|T_E|} 
 \leq \frac{3h_{T_E}^2}{h_E\rho_K} 
 \leq 3c_{\mathcal K}\sigma_{\mathcal K}\frac{h_{T_E}^2}{h_K^2} 
 \leq 3c_{\mathcal K}\sigma_{\mathcal K} 
 = \sigma_{\mathcal T}.
\]
\end{proof}
In the previous proof, we discovered and applied the estimate
\begin{equation}\label{eq:boundAreaOfTE}
 |T_E|\geq \tfrac{1}{2}h_E\rho_K
\end{equation}
for the area of the auxiliary triangle. This inequality will be of importance once more. 
We may also consider the auxiliary triangulation~$\mathcal T_h(\Omega)$ of the whole domain~$\Omega$ which is constructed by gluing the local triangulations~$\mathcal T_h(K)$. Obviously, $\mathcal T_h(\Omega)$ is also shape-regular in the sense of Ciarlet.
Furthermore, the angles in the auxiliary triangulation $\mathcal T_h(K)$ next to~$\partial K$ can be bounded from below. This gives rise to the following result.
\begin{lemma}\label{lem:IsoscelesTriangles}
 Let $\mathcal K_h$ be a regular mesh. Then, there is an angle $\alpha_\mathcal{K}$ with $0<\alpha_\mathcal{K}\leq\pi/3$ such that for all elements $K\in\mathcal K_h$ and all its edges $E\in\mathcal E(K)$ the isosceles triangle~$T_E^\mathrm{iso}$ with longest side $E$ and two interior angles $\alpha_\mathcal{K}$ lies inside $T_E\in\mathcal T_h(K)$ and thus inside the element $K$, see Figure~\ref{fig:regMesh_auxTria}. The angle~$\alpha_\mathcal{K}$ only depends on~$\sigma_\mathcal{K}$ and $c_\mathcal{K}$.
\end{lemma}
\begin{proof}
 Let $T_E\in\mathcal T_h(K)$. We bound the angle~$\beta$ in $T_E$ next to $E=\overline{z_bz_e}$ from below, see Figure~\ref{fig:regMesh_auxTria}. Without loss of generality, we assume that~$\beta<\pi/2$. Using the projection~$y$ of~$z_K$ onto the straight line through the edge~$E$, we recognize 
 \[\sin\beta=\frac{\vert y-z_K\vert}{\vert z_b-z_K\vert}\geq \frac{\rho_K}{h_K} \geq \frac{1}{\sigma_\mathcal{K}}\in(0,1).\]
 Consequently, it is $\beta\geq\arcsin{\sigma_\mathcal{K}^{-1}}$. Since this estimate is valid for all angles next to~$\partial K$ of the auxiliary triangulation, the isosceles triangles~$T_E^\mathrm{iso}$, $E\in\mathcal E(K)$ with common angle $\alpha_\mathcal{K}=\min\{\pi/3, \arcsin\sigma_\mathcal{K}^{-1}\}$ lie inside the auxiliary triangles~$T_E$ and therefore inside~$K$.
\end{proof}

\begin{figure}[htbp]
 \centering
 \begin{tabular}{l@{\qquad}l@{\qquad}l}
 \scalebox{0.8}{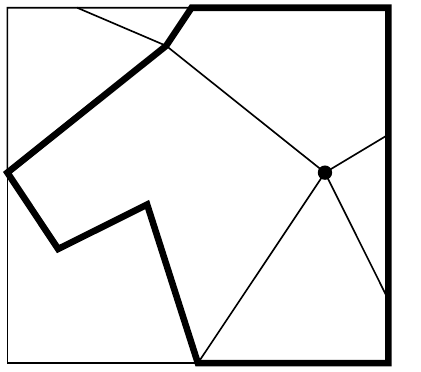} &
 \scalebox{0.8}{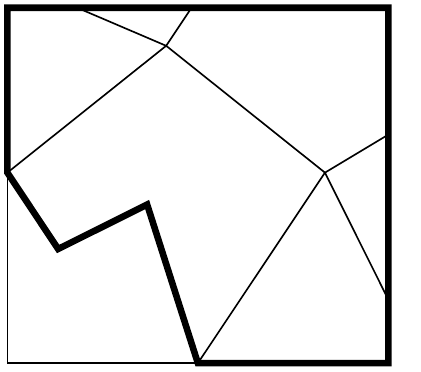} &
 \scalebox{0.8}{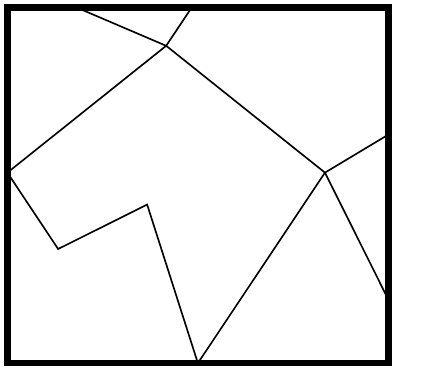} \\
 \scalebox{0.8}{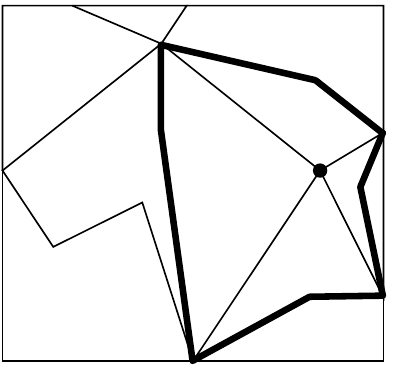} &
 \scalebox{0.8}{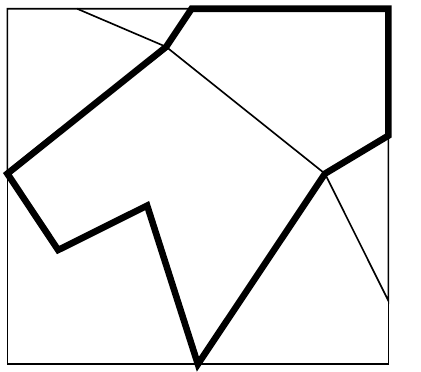} &
 \scalebox{0.8}{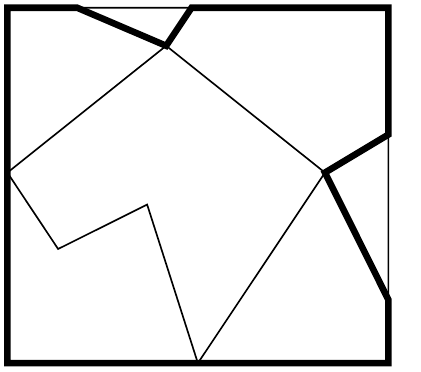}
 \end{tabular}
 \caption{Example of a mesh and neighbourhoods of nodes, edges and elements}
 \label{fig:neighbourhoods}
\end{figure}
We have to consider neighbourhoods of nodes, edges and elements for a proper definition of a quasi-interpolation operator. At the current point, we want to give a preview and prove some properties. These neighbourhoods are open sets and they are defined as element patches by
%\begin{equation}\label{eq:neighbourhoods}
% \begin{gathered}
% \overline\omega_z = \bigcup_{z\in \mathcal N(K')}\overline{K'},\qquad \overline\omega_E = \bigcup_{E\in\mathcal E(K')}\overline{K'},\qquad \overline\omega_K = \bigcup_{\overline K\cap \overline{K'}\neq\varnothing}\overline{K'}\\
% \overline\omega_z = \bigcup_{z\in \mathcal N(K')}\overline{K'},\qquad \overline\omega_E = \bigcup_{E\in\mathcal E(K')}\overline{K'},\qquad \overline\omega_K = \bigcup_{\overline K\cap \overline{K'}\neq\varnothing}\overline{K'}
% \end{gathered}
%\end{equation}
\begin{align}
 \overline\omega_z &= \bigcup_{z\in \mathcal N(K')}\overline{K'}, & 
 \overline\omega_E &= \bigcup_{\mathcal N(E)\cap \mathcal N(K')\neq\varnothing}\hspace{-2ex}\overline{K'}, & 
 \overline\omega_K &= \bigcup_{\mathcal N(K)\cap \mathcal N(K')\neq\varnothing}\hspace{-2ex}\overline{K'}\label{eq:neighbourhoods}\\
 \overline{\widetilde\omega}_z &= \bigcup_{T\in\mathcal T_h(\Omega):z\in\overline T}\hspace{-1ex}\overline{T}, & 
 \overline{\widetilde\omega}_E &= \bigcup_{E\in\mathcal E(K')}\overline{K'}, & 
 \overline{\widetilde\omega}_K &= \bigcup_{\mathcal E(K)\cap \mathcal E(K')\neq\varnothing}\hspace{-2ex}\overline{K'}\label{eq:neighbourhoods2}
\end{align}
for $z\in\mathcal N_h$, $E\in\mathcal E_h$ and $K\in\mathcal K_h$, see Figure~\ref{fig:neighbourhoods}. 
An important role play the neighbourhoods~$\omega_z$ and $\widetilde\omega_z$ of a node~$z$. Their diameters are denoted by~$h_{\omega_z}$ and $h_{\widetilde\omega_z}$. Furthermore, $\omega_z$ is of comparable size to~$K\subset\omega_z$ as shown in 
\begin{lemma}\label{lem:propertiesMesh}
 Let $\mathcal K_h$ be regular. Then, the mesh fulfils:
 \begin{enumerate}
  \item The number of nodes and edges per element is uniformly bounded, \\i.e. $\vert\mathcal N(K)\vert=\vert\mathcal E(K)\vert\leq c,\quad\forall K\in\mathcal K_h$.
  \item Every node belongs to finitely many elements, \\i.e. $\vert\{K\in\mathcal K_h:z\in\mathcal N(K)\}\vert\leq c$ $\forall z\in\mathcal N_h$.
  \item Each element is covered by a uniformly bounded number of neighbourhoods of elements, i.e. $\vert\{K^\prime\in\mathcal K_h:K\subset \omega_{K^\prime}\}\vert \leq c,\quad\forall K\in\mathcal K_h$.
  \item For all $z\in\mathcal N_h$ and $K\subset\omega_z$, it is $h_{\omega_z}\leq ch_K$.
 \end{enumerate}
 The generic constant $c>0$ only depends on $\sigma_\mathcal{K}$ and $c_\mathcal{K}$ from Definition~\ref{def:regularMesh}.
\end{lemma}
\begin{proof}
 The properties \textit{2, 3} and \textit{4} are proven as in~\cite{Weisser2011} and rely on Lemma~\ref{lem:IsoscelesTriangles}, which guaranties that all interior angles of the polygonal elements are bounded from below by some uniform angle~$\alpha_\mathcal{K}$.

 To see the first property, we exploit the regularity of the mesh. Let $K\in\mathcal K_h$. In 2D it is obviously $|\mathcal N(K)|=|\mathcal E(K)|$. With the help of~\eqref{eq:boundAreaOfTE}, we obtain
 \begin{eqnarray*}
  h_K^2|\mathcal N(K)|
  & \leq & \sigma_\mathcal{K}\rho_\mathcal{K}\,h_K|\mathcal E(K)|\\
  & \leq & \sigma_\mathcal{K}\rho_\mathcal{K}\sum_{E\in\mathcal E(K)}c_\mathcal{K}h_E\\
  & \leq & \sigma_\mathcal{K}c_\mathcal{K}\sum_{E\in\mathcal E(K)}2|T_E|\\
  & = & 2\sigma_\mathcal{K}c_\mathcal{K}\,|K|\\
  & \leq & 2\sigma_\mathcal{K}c_\mathcal{K}\,h_K^2
 \end{eqnarray*}
 Consequently, we have $|\mathcal N(K)|\leq 2\sigma_\mathcal{K}c_\mathcal{K}$.

\end{proof}

\subsection{Trefftz-like trial space}
In order to define a conforming, discrete space $V_h^k\subset V$ over the polygonal decomposition~$\mathcal K_h$, we make use of local Trefftz-like trial functions. 
We give a brief review of the approach in~\cite{Weisser2014}, which yields an approximation space of order $k\in\mathbb N$.
The discrete space is constructed by prescribing its basis functions that are subdivided into nodal, edge and element basis functions. Each of them is locally (element-wise) the unique solution of a boundary value problem.

The nodal functions $\psi_z$, $z\in\mathcal N_h$ are defined by
\begin{gather*}
 -\Delta \psi_z=0\quad\mbox{in } K\quad\mbox{for all } K\in\mathcal K_h,\\
 \psi_z(x) = \begin{cases}1 & \mbox{for } x=z, \\ 0 & \mbox{for } x\in\mathcal N_h\setminus\{z\},\end{cases}\\
 \psi_z \mbox{ is linear on each edge of the mesh}.
\end{gather*}
Let $p_{E,0} = \psi_{z_b}|_E$ and $p_{E,1} = \psi_{z_e}|_E$ for $E=\overline{z_bz_e}$ and let $\{p_{E,i}:i=0,\ldots,k\}$ be a basis of the polynomial space~$\mathcal P^k(E)$ of order~$k$ over~$E$. Then the edge bubble functions $\psi_{E,i}$ for $i=2,\ldots,k$, $E\in\mathcal E_h$ are defined by
\begin{gather*}
 -\Delta \psi_{E,i}=0\quad\mbox{in } K\quad\mbox{for all } K\in\mathcal K_h,\\
 \psi_{E,i} = \begin{cases}p_{E,i} & \mbox{on } E, \\ 0 & \mbox{on } \mathcal E_h\setminus\{E\}.\end{cases}
\end{gather*}
Finally, the element bubble functions $\psi_{K,i,j}$ for $i=0,\ldots,k-2$ and $j=0,\ldots,i$, $K\in\mathcal K_h$ are defined by
\begin{equation*}
 \begin{aligned}
  -\Delta\psi_{K,i,j} &= p_{K,i,j} &&\mbox{in } K,\\
  \psi_{K,i,j} &= 0 &&\mbox{else},
 \end{aligned}
\end{equation*}
where $\{p_{K,i,j}: i=0,\ldots,k-2 \mbox{ and } j=0,\ldots,i\}$ is a basis of the polynomial space~$\mathcal P^{k-2}(K)$ of order~$k-2$ over~$K$.

Due to the regularity of the local problems each basis function~$\psi$ fulfills $\psi\in C^2(K)\cap C^0(\overline K)$ for convex $K\in\mathcal K_h$. In the case of non-convex elements, the definitions are understood in the weak sense, but we still have $\psi\in H^1(K)\cap C^0(\overline K)$. Consequently, it is $\psi\in H^1(\Omega)$ and we set
\[V_h^k = V_{h,1}^k\oplus V_{h,2}^k \subset V,\]
where
\[V_{h,1}^k = \Span\{\psi_z, \psi_{E,i}: i=2,\ldots,k, z\in\mathcal N_h\setminus\mathcal N_{h,D}, E\in\mathcal E_h\setminus\mathcal E_{h,D}\}\]
contains the locally (weakly) harmonic trial functions and
\[V_{h,2}^k = \Span\{\psi_{K,i,j}: i=0,\ldots,k-2 \mbox{ and } j=0,\ldots,i, K\in\mathcal K_h\}\]
contains the trial functions vanishing on all edges. It can be seen that the restriction of $V_h^k$ to an element $K\in\mathcal K_h$, which does not touch the Dirichlet boundary, fulfills
\[V_h^k|_K = \{v\in H^1(K): \Delta v\in\mathcal P^{k-2}(K)\mbox{ and } v|_{\partial K}\in\mathcal P_\mathrm{pw}^k(\partial K)\},\]
where 
\[\mathcal P_\mathrm{pw}^k(\partial K)=\{p\in C^0(\partial K): p|_E\in\mathcal P^k(E), E\in\mathcal E(K)\}.\]

\subsection{Discrete variational formulation}
\label{subsec:DiscVF}
With the help of the conforming approximation space~$V_h^k\subset V$ over polygonal meshes, we can give the discrete version of the variational formulation~\eqref{eq:VF}. Thus, we obtain
\begin{equation}\label{eq:discVF}
 \mbox{Find } u_h\in g_D+V_h^k:\quad b(u_h,v_h) = (f,v_h)_\Omega+(g_{N},v_h)_{\Gamma_N}\quad \forall v_h\in V_h^k.
\end{equation}
For simplicity, we assume that $g_D\in\mathcal P_\mathrm{pw}^k(\partial \Omega)$, such that its extension into $H^1(\Omega)$ can be chosen as interpolation using nodal and edge basis functions.

A closer look at the trial functions enables a simplification of~\eqref{eq:discVF}. Since the nodal and edge basis functions~$\psi$ are (weakly) harmonic on each element, they fulfill
\begin{equation}\label{eq:WeaklyHarmonic}
 (\nabla \psi,\nabla v)_K = 0 \quad\forall v\in H_0^1(K),
\end{equation}
and especially $(\nabla \psi,\nabla \psi_K)_K = 0$ for $\psi_K\in V_{h,2}^k$, because of $\psi_K\in H^1_0(K)$. Since the diffusion coefficient is assumed to be piecewise constant, the variational formulation decouples. Thus, the discrete problem~\eqref{eq:discVF} with solution $u_h=u_{h,1}+u_{h,2}\in V_h^k$ is equivalent to
\begin{equation}\label{eq:discVF1}
 \mbox{Find } u_{h,1}\in g_D+V_{h,1}^k:\quad b(u_{h,1},v_h) = (f,v_h)_\Omega+(g_{N},v_h)_{\Gamma_N}\quad \forall v_h\in V_{h,1}^k,
\end{equation}
and
\begin{equation}\label{eq:discVF2}
 \mbox{Find } u_{h,2}\in V_{h,2}^k:\quad b(u_{h,2},v_h) = (f,v_h)_\Omega\quad \forall v_h\in V_{h,2}^k.
\end{equation}
Moreover, \eqref{eq:discVF2} reduces to a problem on each element since the support of the element bubble functions are restricted to one element. Consequently, $u_{h,2}$ can be computed in a preprocessing step on an element level. A further observation is, that in the case of vanishing source term~$f$, equation~\eqref{eq:discVF2} yields $u_{h,2}=0$. Therefore, it is sufficient to consider the discrete variational formulation~\eqref{eq:discVF1} for the homogeneous diffusion equation.

In~\cite{Weisser2014}, it has been shown that the discrete variational formulation with trial space $V_h^k$ yields optimal rates of convergence for uniform mesh refinement under classical regularity assumptions on the boundary value problem. More precisely, we have
\[\Vert u-u_h\Vert_{1,\Omega} \leq c\,h^k\,\vert u\vert_{k+1,\Omega} \quad\mbox{for } u\in H^{k+1}(\Omega),\]
and additionally, for $H^2$-regular problems,
\[\Vert u-u_h\Vert_{0,\Omega} \leq c\,h^{k+1}\,\vert u\vert_{k+1,\Omega} \quad\mbox{for } u\in H^{k+1}(\Omega),\]
where $h=\max\{h_K:K\in\mathcal K_h\}$ and the constant~$c$ only depends on the mesh parameters given in Definition~\ref{def:regularMesh} as well as on~$k$.

\subsection{BEM-based approximation}
\label{subsec:BEMbasedApproximationI}
In this subsection we briefly discuss how to deal with the implicitly defined trial functions. More details can be found in the publications~\cite{RjasanowWeisser2012,Weisser2014}, for example. We focus on a local problem of the following form
\begin{equation}\label{eq:Laplace}
 \begin{aligned}
  -\Delta\psi &= p_K &&\mbox{in } K,\\
  \psi &= p_{\partial K} &&\mbox{on } \partial K,
 \end{aligned}
\end{equation}
where $p_K\in\mathcal P^{k-2}(K)$, $p_{\partial K}\in\mathcal P_\mathrm{pw}^k(\partial K)$ and $K\in\mathcal K_h$ is an arbitrary element. 

Let $\gamma_0^K:H^1(K)\to H^{1/2}(\partial K)$ be the trace operator and denote by $\gamma_1^K$ the conormal derivative, which maps the solution of~\eqref{eq:Laplace} to its Neumann trace on the boundary~$\partial K$, see~\cite{McLean2000}. For $v\in H^1(K)$ with $\Delta v$ in the dual of $H^1(K)$, $\gamma_1^Kv$ is defined as unique function in $H^{-1/2}(\partial K)$ which satisfies Greens first identity such that
\begin{equation*}%\label{eq:GreensFirstIdentity}
  \int_{K}\nabla v\cdot\nabla w = 
   \int_{\partial K}\gamma_1^K v\gamma_0^Kw - \int_{K}w\Delta v
\end{equation*}
for $w\in H^1(K)$. If $v$ is smooth, e.g. $v\in H^2(K)$, we have
\[
 \gamma_0^Kv(x) = v(x)
 \quad\mbox{and}\quad
 \gamma_1^Kv(x) = n_K(x)\cdot\gamma_0^K\nabla v(x)
 \quad\mbox{for } x\in\partial K,
\]
where $n_K(x)$ denotes the outer normal vector of the domain $K$ at $x$.  

Without loss of generality, we only consider the Laplace equation, i.e. $p_K=0$ in~\eqref{eq:Laplace}. In the case of a general $\psi\in V_h^k$ with $-\Delta\psi=p_K$, we can always construct a polynomial $q\in\mathcal P^k(K)$, see~\cite{KarachikAntropova2010}, such that $-\Delta(\psi-q)=0$ in~$K$ and $\psi-q\in\mathcal P_\mathrm{pw}^k(\partial K)$ on~$\partial K$. Due to the linearity of $\gamma_1^K$, it is
\[\gamma_1^K\psi = \gamma_1^K(\psi-q) + \gamma_1^Kq\]
with $\gamma_1^Kq\in\mathcal P_\mathrm{pw}^{k-1}(\partial K)$. Thus, $\gamma_1^K\psi$ can be expressed as conormal derivative of a (weakly) harmonic function and a piecewise polynomial. Consequently, we reduced the general case to the Laplace problem.

It is well known, that the solution of~\eqref{eq:Laplace} with $p_K=0$ has the representation
\begin{equation}\label{eq:RepresentationFormula}
 \psi(x)=\int_{\partial K}U^*(x,y)\gamma_1^K \psi(y)\,ds_y - \int_{\partial K}\!\gamma_{1,y}^K U^*(x,y) \gamma_0^K \psi(y) \,ds_y \quad\mbox{for } x\in K,
\end{equation}
where $U^*$ is the fundamental solution of minus the Laplacian with
\[U^*(x,y) = -\frac{1}{2\pi} \ln|x-y|\quad \mbox{for } x,y\in\mathbb{R}^2,\]
see, e.g.,~\cite{McLean2000}. 
The Dirichlet trace~$\gamma_0^K\psi=p_{\partial K}$ is known by the problem, whereas the Neumann trace~$\gamma_1^K\psi$ is an unknown quantity. Taking the trace of~\eqref{eq:RepresentationFormula} yields the boundary integral equation
\begin{equation}\label{eq:BIE}
 \mathbf V_K\gamma_1^K\psi = \left(\tfrac{1}{2}\mathbf I+\mathbf K_K\right)p_{\partial K},
\end{equation}
with the single-layer potential operator
\[(\mathbf V_K \vartheta)(x) = \gamma^K_0\int_{\partial K}\!U^*(x,y)\vartheta(y)\,ds_y \quad\mbox{for }\vartheta\in H^{-1/2}(\partial K),\]
and the double-layer potential operator
\[(\mathbf K_K \xi)(x) = \lim_{\varepsilon\to 0}\int_{y\in\partial K:|y-x|\geq\varepsilon}\hspace{-0.7cm}\gamma_{1,y}^K U^*(x,y)\xi(y)\,ds_y \quad\mbox{for } \xi\in H^{1/2}(\partial K),\]
where $x\in\partial K$. To obtain an approximation of the Neumann trace a Galerkin scheme is applied to~\eqref{eq:BIE}, which only involves integration over the boundary of the element, see, e.g.,~\cite{Steinbach2007}. Thus, we approximate $\gamma_1^K\psi$ by $\widetilde{\gamma_1^K\psi}\in\mathcal P_\mathrm{pw,d}^{k-1}(\partial K)$ such that
\begin{equation}\label{eq:discVFofBIE}
 \left(\mathbf V_K\widetilde{\gamma_1^K\psi},q\right)_{\partial K} = \left(\left(\tfrac{1}{2}\mathbf I+\mathbf K_K\right)p_{\partial K},q\right)_{\partial K} \quad\forall q\in\mathcal P_\mathrm{pw,d}^{k-1}(\partial K),
\end{equation}
where
\[\mathcal P_\mathrm{pw,d}^{k-1}(\partial K)=\{q\in L_2(\partial K): q|_E\in\mathcal P^{k-1}(E), E\in\mathcal E(K)\}.\]
The discrete formulation~\eqref{eq:discVFofBIE} has a unique solution. The approximation described here is very rough, since no further discretization of the edges is performed. For $k=1$, the Neumann trace is approximated by a constant on each edge~$E\in\mathcal E(K)$, for example. The resulting boundary element matrices are dense, but, they are rather small, since the number of edges per element is bounded by Lemma~\ref{lem:propertiesMesh}. Consequently, the solution of the system of linear equations is approximated by an efficient LAPACK routine. 
The boundary element matrices of different elements are independent of each other. Therefore, they are computed in parallel during a preprocessing step for the overall simulation. Furthermore, once computed, they are used throughout the computations.

\subsection{Approximated discrete variational formulation}
\label{subsec:ApproxDiscVF}
It remains to discuss the approximations of the terms
\[ b(\psi,\varphi), \quad (f,\varphi)_\Omega \quad\mbox{and}\quad (g_N,\varphi)_{\Gamma_N} \quad\mbox{for } \psi,\varphi\in V_h^k\]
%\[ (g_N,\varphi)_{\Gamma_N}, \quad (f,\varphi)_\Omega \quad\mbox{and}\quad b(\psi,\varphi) \quad\mbox{for } \psi,\varphi\in V_h^k\]
in the variational formulation~\eqref{eq:discVF1}-\eqref{eq:discVF2}. 
For the approximation of the first term~$b(\psi,\varphi)$, we apply Greens identity locally. Let $\psi\in V_h^k$, $v\in V$ and $a(x) = a_K$ on $K\in\mathcal K_h$. Greens first identity yields
\begin{equation}\label{eq:ExactBilinearForm}
 b(\psi,v) 
 = \sum_{K\in\mathcal K_h}a_K(\nabla\psi,\nabla v)_K 
 = \sum_{K\in\mathcal K_h}a_K\left\{(\gamma_1^K\psi,\gamma_0^K v)_{\partial K}
                                    -(\Delta\psi,v)_K\right\}.
\end{equation}
The approximation of this bilinear form is defined by
\begin{equation}\label{eq:ApproxBilinearForm}
 b_h(\psi,v) 
 = \sum_{K\in\mathcal K_h}a_K\left\{(\widetilde{\gamma_1^K\psi},\gamma_0^K v)_{\partial K}
                                    -(\Delta\psi,\widetilde{v})_K\right\},
\end{equation}
where $\widetilde{\gamma_1^K\psi}\in\mathcal P_\mathrm{pw,d}^{k-1}(\partial K)$ is the approximation coming from the BEM and $\widetilde v$ is an appropriate polynomial approximation of $v$ over $K$ as, i.e., the averaged Taylor polynomial in Lemma~(4.3.8) of~\cite{BrennerScott2002}. If $b_h(\psi,\varphi)$ is evaluated for $\psi,\varphi\in V_{h,1}^k$, we obviously end up with only boundary integrals, where polynomials are integrated. For $\psi,\varphi\in V_{h,2}^k$, the volume integrals remain. These integrals can also be evaluated analytically, since the integrand is a polynomial of order~$2k-2$.

The volume integral $(f,\varphi)_\Omega$ is treated in a similar way as $(\Delta\psi,v)_K$. Let $\widetilde\varphi$ be a piecewise polynomial approximation of~$\varphi$, then we use as approximation $(f,\widetilde\varphi)_\Omega$. The resulting integral is treated over each polygonal element by numerical quadrature. It is possible to utilize an auxiliary triangulation for the numerical integration over the elements or to apply a quadrature scheme for polygonal domains directly, see~\cite{MousaviSukumar2011}. 

The last integral~$(g_N,\varphi)_{\Gamma_N}$ is treated by means of Gaussian quadrature over each edge within the Neumann boundary. The data~$g_N$ is given and the trace of the shape functions is known explicitly by their definition. 

Finally, the approximated discrete variational formulation for $u_h=u_{h,1}+u_{h,2}\in V_h^k$ is given as
\begin{equation}\label{eq:approxDiscVF1}
 \mbox{Find } u_{h,1}\in g_D+V_{h,1}^k:\quad b_h(u_{h,1},v_h) = (f,\widetilde v_h)_\Omega+(g_{N},v_h)_{\Gamma_N}\quad \forall v_h\in V_{h,1}^k,
\end{equation}
and
\begin{equation}\label{eq:approxDiscVF2}
 \mbox{Find } u_{h,2}\in V_{h,2}^k:\quad b_h(u_{h,2},v_h) = (f,\widetilde v_h)_\Omega\quad \forall v_h\in V_{h,2}^k.
\end{equation}

\begin{remark}
The approximation of $(\Delta\psi,v)_K$ by $(\Delta\psi,\widetilde v)_K$ and $(f,v)_K$ by $(f,\widetilde v)_K$ with a polynomial~$\widetilde v$ can be interpreted as numerical quadrature. In consequence, we can directly approximate $(\Delta\psi,v)_K$ by an appropriate quadrature rule in the computational realization without the explicit construction of~$\widetilde v$. In this case the representation formula~\eqref{eq:RepresentationFormula} is used to evaluate the shape functions inside the elements. 
\end{remark}

\begin{remark}
The representation~\eqref{eq:ApproxBilinearForm} is advantageous for the a posteriori error analysis in the following, where we indeed consider $b_h(\psi,v)$ for $v\in V$. In the computational realization, however, we are only interested in $v=\varphi\in V_h^k$. Consequently, one might use for $\psi,\varphi\in V_{h,2}^k$ with $\supp(\psi)=\supp(\varphi)=K$ homogenization. Therefore, let $q\in\mathcal P^{k}(K)$ such that $\Delta\varphi=\Delta q$. This time, Greens identity yields
\[
 b(\psi,\varphi) 
 = a_K\left((\gamma_1^K\psi,\gamma_0^Kq)_{\partial K} - (\Delta\psi,q)_K\right),
\]
since $(\nabla(\varphi-q),\nabla\psi)_K = 0$ due to~\eqref{eq:WeaklyHarmonic}. In this formulation, the volume integral can be evaluated analytically and there is no need to approximate it. But, we have used the properties of the basis functions. For the approximation of the boundary integral $(\gamma_1^K\psi,\gamma_0^Kq)_{\partial K}$ one might even use an improved strategy compared to the one described above, see~\cite{RjasanowWeisser2012}.
\end{remark}

%% file: fig/starElem.pdf_t
\begin{picture}(0,0)%
\includegraphics{starElem.pdf}%
\end{picture}%
\setlength{\unitlength}{4144sp}%
\begingroup\makeatletter\ifx\SetFigFont\undefined%
\gdef\SetFigFont#1#2#3#4#5{%
  \reset@font\fontsize{#1}{#2pt}%
  \fontfamily{#3}\fontseries{#4}\fontshape{#5}%
  \selectfont}%
\fi\endgroup%
\begin{picture}(3238,2156)(-493,-1364)
\put(901,-826){\makebox(0,0)[lb]{\smash{{\SetFigFont{12}{14.4}{\familydefault}{\mddefault}{\updefault}{\color[rgb]{0,0,0}$z_K$}%
}}}}
\put(451,-556){\makebox(0,0)[lb]{\smash{{\SetFigFont{12}{14.4}{\familydefault}{\mddefault}{\updefault}{\color[rgb]{0,0,0}$\rho_K$}%
}}}}
\put(316,569){\makebox(0,0)[lb]{\smash{{\SetFigFont{12}{14.4}{\familydefault}{\mddefault}{\updefault}{\color[rgb]{0,0,0}$z_b$}%
}}}}
\put(316,299){\makebox(0,0)[lb]{\smash{{\SetFigFont{12}{14.4}{\familydefault}{\mddefault}{\updefault}{\color[rgb]{0,0,0}$E$}%
}}}}
\put(316, 29){\makebox(0,0)[lb]{\smash{{\SetFigFont{12}{14.4}{\familydefault}{\mddefault}{\updefault}{\color[rgb]{0,0,0}$z_e$}%
}}}}
\put(2341,-16){\makebox(0,0)[lb]{\smash{{\SetFigFont{12}{14.4}{\familydefault}{\mddefault}{\updefault}{\color[rgb]{0,0,0}$K$}%
}}}}
\end{picture}%

%% file: fig/starElem_auxTria_angle.pdf_t
\begin{picture}(0,0)%
\includegraphics{starElem_auxTria_angle.pdf}%
\end{picture}%
\setlength{\unitlength}{4144sp}%
\begingroup\makeatletter\ifx\SetFigFont\undefined%
\gdef\SetFigFont#1#2#3#4#5{%
  \reset@font\fontsize{#1}{#2pt}%
  \fontfamily{#3}\fontseries{#4}\fontshape{#5}%
  \selectfont}%
\fi\endgroup%
\begin{picture}(3463,2402)(182,-1610)
\put(901,-826){\makebox(0,0)[lb]{\smash{{\SetFigFont{12}{14.4}{\familydefault}{\mddefault}{\updefault}{\color[rgb]{0,0,0}$z_K$}%
}}}}
\put(451,-556){\makebox(0,0)[lb]{\smash{{\SetFigFont{12}{14.4}{\familydefault}{\mddefault}{\updefault}{\color[rgb]{0,0,0}$\rho_K$}%
}}}}
\put(2656,-61){\makebox(0,0)[lb]{\smash{{\SetFigFont{12}{14.4}{\familydefault}{\mddefault}{\updefault}{\color[rgb]{0,0,0}$E$}%
}}}}
\put(3556,-151){\makebox(0,0)[lb]{\smash{{\SetFigFont{12}{14.4}{\familydefault}{\mddefault}{\updefault}{\color[rgb]{0,0,0}$z_b$}%
}}}}
\put(1981,-16){\makebox(0,0)[lb]{\smash{{\SetFigFont{12}{14.4}{\familydefault}{\mddefault}{\updefault}{\color[rgb]{0,0,0}$z_e$}%
}}}}
\put(995,-227){\makebox(0,0)[lb]{\smash{{\SetFigFont{12}{14.4}{\familydefault}{\mddefault}{\updefault}{\color[rgb]{0,0,0}$h_{a}$}%
}}}}
\put(1936,-376){\makebox(0,0)[lb]{\smash{{\SetFigFont{12}{14.4}{\familydefault}{\mddefault}{\updefault}{\color[rgb]{0,0,0}$T_E$}%
}}}}
\put(2341,-601){\makebox(0,0)[lb]{\smash{{\SetFigFont{12}{14.4}{\familydefault}{\mddefault}{\updefault}{\color[rgb]{0,0,0}$h_{T_E}$}%
}}}}
\put(2881,-1546){\makebox(0,0)[lb]{\smash{{\SetFigFont{12}{14.4}{\familydefault}{\mddefault}{\updefault}{\color[rgb]{0,0,0}$K$}%
}}}}
\put(3286, 74){\makebox(0,0)[lb]{\smash{{\SetFigFont{12}{14.4}{\familydefault}{\mddefault}{\updefault}{\color[rgb]{0,0,0}$\beta$}%
}}}}
\put(1036,119){\makebox(0,0)[lb]{\smash{{\SetFigFont{12}{14.4}{\familydefault}{\mddefault}{\updefault}{\color[rgb]{0,0,0}$y$}%
}}}}
\end{picture}%

%% file: fig/node_nb.pdf_t
\begin{picture}(0,0)%
\includegraphics{node_nb.pdf}%
\end{picture}%
\setlength{\unitlength}{4144sp}%
\begingroup\makeatletter\ifx\SetFigFont\undefined%
\gdef\SetFigFont#1#2#3#4#5{%
  \reset@font\fontsize{#1}{#2pt}%
  \fontfamily{#3}\fontseries{#4}\fontshape{#5}%
  \selectfont}%
\fi\endgroup%
\begin{picture}(1909,1697)(-3,-854)
\put(1441,164){\makebox(0,0)[lb]{\smash{{\SetFigFont{12}{14.4}{\familydefault}{\mddefault}{\updefault}{\color[rgb]{0,0,0}$z$}%
}}}}
\put(1891,-781){\makebox(0,0)[lb]{\smash{{\SetFigFont{12}{14.4}{\familydefault}{\mddefault}{\updefault}{\color[rgb]{0,0,0}$\omega_z$}%
}}}}
\end{picture}%

%% file: fig/edge_nb.pdf_t
\begin{picture}(0,0)%
\includegraphics{edge_nb.pdf}%
\end{picture}%
\setlength{\unitlength}{4144sp}%
\begingroup\makeatletter\ifx\SetFigFont\undefined%
\gdef\SetFigFont#1#2#3#4#5{%
  \reset@font\fontsize{#1}{#2pt}%
  \fontfamily{#3}\fontseries{#4}\fontshape{#5}%
  \selectfont}%
\fi\endgroup%
\begin{picture}(1909,1697)(-3,-854)
\put(1891,-781){\makebox(0,0)[lb]{\smash{{\SetFigFont{12}{14.4}{\familydefault}{\mddefault}{\updefault}{\color[rgb]{0,0,0}$\omega_E$}%
}}}}
\put(1126,389){\makebox(0,0)[lb]{\smash{{\SetFigFont{12}{14.4}{\familydefault}{\mddefault}{\updefault}{\color[rgb]{0,0,0}$E$}%
}}}}
\end{picture}%

%% file: fig/elem_nb.pdf_t
\begin{picture}(0,0)%
\includegraphics{elem_nb.pdf}%
\end{picture}%
\setlength{\unitlength}{4144sp}%
\begingroup\makeatletter\ifx\SetFigFont\undefined%
\gdef\SetFigFont#1#2#3#4#5{%
  \reset@font\fontsize{#1}{#2pt}%
  \fontfamily{#3}\fontseries{#4}\fontshape{#5}%
  \selectfont}%
\fi\endgroup%
\begin{picture}(1909,1697)(-3,-854)
\put(901,-61){\makebox(0,0)[lb]{\smash{{\SetFigFont{12}{14.4}{\familydefault}{\mddefault}{\updefault}{\color[rgb]{0,0,0}$K$}%
}}}}
\put(1891,-781){\makebox(0,0)[lb]{\smash{{\SetFigFont{12}{14.4}{\familydefault}{\mddefault}{\updefault}{\color[rgb]{0,0,0}$\omega_K$}%
}}}}
\end{picture}%

%% file: fig/node_nb_tilde.pdf_t
\begin{picture}(0,0)%
\includegraphics{node_nb_tilde.pdf}%
\end{picture}%
\setlength{\unitlength}{4144sp}%
\begingroup\makeatletter\ifx\SetFigFont\undefined%
\gdef\SetFigFont#1#2#3#4#5{%
  \reset@font\fontsize{#1}{#2pt}%
  \fontfamily{#3}\fontseries{#4}\fontshape{#5}%
  \selectfont}%
\fi\endgroup%
\begin{picture}(1888,1669)(18,-847)
\put(1891,-511){\makebox(0,0)[lb]{\smash{{\SetFigFont{12}{14.4}{\familydefault}{\mddefault}{\updefault}{\color[rgb]{0,0,0}$\widetilde{\omega}_z$}%
}}}}
\put(1441,164){\makebox(0,0)[lb]{\smash{{\SetFigFont{12}{14.4}{\familydefault}{\mddefault}{\updefault}{\color[rgb]{0,0,0}$z$}%
}}}}
\end{picture}%

%% file: fig/edge_nb_tilde.pdf_t
\begin{picture}(0,0)%
\includegraphics{edge_nb_tilde.pdf}%
\end{picture}%
\setlength{\unitlength}{4144sp}%
\begingroup\makeatletter\ifx\SetFigFont\undefined%
\gdef\SetFigFont#1#2#3#4#5{%
  \reset@font\fontsize{#1}{#2pt}%
  \fontfamily{#3}\fontseries{#4}\fontshape{#5}%
  \selectfont}%
\fi\endgroup%
\begin{picture}(1909,1691)(-3,-848)
\put(1171,389){\makebox(0,0)[lb]{\smash{{\SetFigFont{12}{14.4}{\familydefault}{\mddefault}{\updefault}{\color[rgb]{0,0,0}$E$}%
}}}}
\put(1891,254){\makebox(0,0)[lb]{\smash{{\SetFigFont{12}{14.4}{\familydefault}{\mddefault}{\updefault}{\color[rgb]{0,0,0}$\widetilde{\omega}_E$}%
}}}}
\end{picture}%

%% file: fig/elem_nb_tilde.pdf_t
\begin{picture}(0,0)%
\includegraphics{elem_nb_tilde.pdf}%
\end{picture}%
\setlength{\unitlength}{4144sp}%
\begingroup\makeatletter\ifx\SetFigFont\undefined%
\gdef\SetFigFont#1#2#3#4#5{%
  \reset@font\fontsize{#1}{#2pt}%
  \fontfamily{#3}\fontseries{#4}\fontshape{#5}%
  \selectfont}%
\fi\endgroup%
\begin{picture}(1909,1697)(-3,-854)
\put(901,-61){\makebox(0,0)[lb]{\smash{{\SetFigFont{12}{14.4}{\familydefault}{\mddefault}{\updefault}{\color[rgb]{0,0,0}$K$}%
}}}}
\put(1891,-781){\makebox(0,0)[lb]{\smash{{\SetFigFont{12}{14.4}{\familydefault}{\mddefault}{\updefault}{\color[rgb]{0,0,0}$\widetilde{\omega}_K$}%
}}}}
\end{picture}%

%% file: QuasiInterpolation.tex
\section{Quasi-interpolation}
\label{sec:QuasiInterpolation}
In the case of smooth functions like in $H^2(\Omega)$, it is possible to use nodal interpolation. Such interpolation operators are constructed and studied in~\cite{Weisser2014}, and they yield optimal orders of approximation. The goal of this section, however, is to define interpolation for general functions in $H^1(\Omega)$. Consequently, quasi-interpolation operators are applied, which utilizes the neighbourhoods $\omega_z$ and $\widetilde\omega_z$ defined in~\eqref{eq:neighbourhoods} and \eqref{eq:neighbourhoods2}. Let $\omega\subset\Omega$ be a simply connected domain. With the help of the $L_2$-projection $Q_\omega:L_2(\omega)\to\mathbb R$ into the space of constants, we introduce the quasi-interpolation operators $\mathfrak I_h,\widetilde{\mathfrak I}_h:V\to V_h^1$ as
\[
 \mathfrak I_hv = \sum_{z\in\mathcal N_h\setminus\mathcal N_{h,D}}(Q_{\omega_z}v)\psi_z 
 \qquad\mbox{and}\qquad
 \widetilde{\mathfrak I}_hv = \sum_{z\in\mathcal N_h\setminus\mathcal N_{h,D}}(Q_{\widetilde\omega_z}v)\psi_z 
\]
for $v\in V$. The definition of $\mathfrak I_h$ is a direct generalization of the quasi-interpolation operator in~\cite{Weisser2011} and $\widetilde{\mathfrak I}_h$ is a small modification.
The definition is very similar to the one of Cl\'ement~\cite{Clement1975}. The major difference is the use of non-polynomial trial functions on polygonal meshes with non-convex elements. 
The quasi-interpolation operator~$\mathfrak I_h$ has been studied in~\cite{Weisser2011} on polygonal meshes with convex elements. In the following, the theory is generalized to meshes with star-shaped elements, which are in general non-convex.

Before we give approximation results for the quasi-interpolation operators, some auxiliary results are reviewed and extended. If no confusion arises, we write $v$ for both the function and the trace of the function on an edge. Since Lemma~\ref{lem:IsoscelesTriangles} guaranties the existence of the isosceles triangles with common angles for non-convex elements in a regular mesh, we can directly use the following lemma proven in~\cite{Weisser2011}.
\begin{lemma}\label{lem:step2}
 Let $\mathcal K_h$ be a regular mesh, $v\in H^1(K)$ for $K\in\mathcal K_h$ and $E\in\mathcal E(K)$. Then it holds
 \[\|v\|_{0,E} \leq c\left\{h_E^{-1/2}\|v\|_{0,T_E^\mathrm{iso}}+h_E^{1/2}\vert v\vert_{1,T_E^\mathrm{iso}}\right\}\]
 with the isosceles triangle $T_E^\mathrm{iso}\subset K$ from Lemma~\ref{lem:IsoscelesTriangles}, where $c$ only depends on~$\alpha_\mathcal{K}$, and thus, on the regularity parameters $\sigma_\mathcal{K}$ and $c_\mathcal{K}$.
\end{lemma}

Next, we prove an approximation property for the $L_2$-projection.
\begin{lemma}\label{lem:step1}
 Let $\mathcal K_h$ be a regular mesh. There exist uniform constants~$c$, which only depend on the regularity parameters $\sigma_\mathcal{K}$ and $c_\mathcal{K}$, such that for every $z\in\mathcal N_h$, it holds
 \[ \|v-Q_{\omega_z}v\|_{0,\omega_z} \leq ch_{\omega_z}\vert v\vert_{1,\omega_z} \quad\mbox{for } v\in H^1(\omega_z),\]
and
 \[ \|v-Q_{\widetilde\omega_z}v\|_{0,\widetilde\omega_z} \leq ch_{\widetilde\omega_z}\vert v\vert_{1,\widetilde\omega_z} \quad\mbox{for } v\in H^1(\widetilde\omega_z),\]
 where $h_{\omega_z}$ and $h_{\widetilde\omega_z}$ denote the diameter of $\omega_z$ and $\widetilde\omega_z$, respectively.
\end{lemma}
This result is of interest on its own. It is known that these inequalities hold with the Poincar\'e constant
 \[C_P(\omega) = \sup_{v\in H^1(\omega)} \frac{\Vert v-Q_\omega v\Vert_{0,\omega}}{h_{\omega} \vert v\vert_{1,\omega}} < \infty,\]
which depends on the shape of $\omega$, see~\cite{VeeserVerfuerth2010}. For convex $\omega$, the authors of~\cite{PayneWeinberger1960} showed $C_P(\omega)<1/\pi$. In our situation, however, $\omega_z$ is a patch of non-convex elements which is itself non-convex in general. The lemma says, that we can bound $C(\omega_z)$ and $C(\widetilde\omega_z)$ by a uniform constant~$c$ depending only on the regularity parameters of the mesh.
The main tool in the forthcoming proof is Proposition~2.10~(Decomposition) of~\cite{VeeserVerfuerth2010}. As preliminary of this proposition, an admissible decomposition $\{\omega_i\}_{i=1}^n$ of $\omega$ with pairwise disjoint domains $\omega_i$ and 
 \[\overline\omega = \bigcup_{i=1}^n \overline\omega_i\]
 is needed. Admissible means in this context, that there exist triangles $\{T_i\}_{i=1}^n$ such that $T_i\subset\omega_i$ and for every pair $i,j$ of different indices, there is a sequence $i=k_0,\ldots,k_\ell=j$ of indices such that for every $m$ the triangles $T_{k_{m-1}}$ and $T_{k_m}$ share a complete side. Under these assumptions, the Poincar\'e constant of $\omega$ is bounded by
 \[C_P(\omega) \leq \max_{1\leq i\leq n}\left\{ 8(n-1)\left(1-\min_{1\leq j\leq n}\frac{\vert \omega_j\vert}{\vert\omega\vert}\right) \left(C_P^2(\omega_i)+2C_P(\omega_i)\right) \frac{\vert\omega\vert\,h_{\omega_i}^2}{\vert T_i\vert\,h_{\omega}^2}\right\}^{1/2}.\]
\begin{proof}[Proof (Lemma~\ref{lem:step1})]
The second estimate in the lemma can be seen easily. The neighbourhood $\widetilde\omega_z$ is a patch of triangles, see~\eqref{eq:neighbourhoods2}. Thus, we choose $\omega_i=T_i$, $i=1,\ldots,n$ with $\{T_i\}_{i=1}^n=\{T\in\mathcal T_h(\Omega): z\in\overline T\}$ for the admissible decomposition of~$\widetilde\omega_z$. According to Lemma~\ref{lem:propertiesMesh}, $n$ is uniformly bounded. Furthermore, it is $C_P(\omega_i)<1/\pi$, $|\widetilde\omega_z|<h_{\widetilde\omega_z}^2$ and $h_{\omega_i}^2/|T_i|=h_{T_i}^2/|T_i|\leq c$, because of the shape-regularity of the auxiliary triangulation proven in Lemma~\ref{lem:regAuxTria}. Consequently, applying Proposition~2.10~(Decomposition) of~\cite{VeeserVerfuerth2010} yields $C_P(\widetilde\omega_z)<c$ that proves the second estimate in the lemma.

The first estimate is proven in three steps. First, consider an element~$K$, which fulfills the regularity assumptions of Definition~\ref{def:regularMesh}. Remembering the construction of $\mathcal T_h(K)$, $K$ can be interpreted as patch of triangles~$\widetilde\omega_{z_K}$ corresponding to the point~$z_K$. Arguing as above gives $C_P(K)<c$.

Next, we consider a neighbourhood~$\omega_z$ and apply Proposition~2.10 of~\cite{VeeserVerfuerth2010}. In the second step, to simplify the explanations, we assume that the patch consists of only one element, i.e. $\omega_z=K\in\mathcal K_h$, and let $E_1,E_2\in\mathcal E(K)$ with $z=\overline{E}_1\cap\overline{E}_2$. We decompose $\omega_z$, or equivalently $K$, into $\omega_1$ and $\omega_2$ such that $n=2$. The decomposition is done by splitting $K$ along the polygonal chain through the points $z$, $z_K$ and $z'$, where $z'\in\mathcal N(K)$ is chosen such that the angle $\beta=\angle zz_Kz'$ is maximized, see Figure~\ref{fig:admissibleDecomposition} left. It is $\beta\in(\pi/2,\pi]$, since $K$ is star-shaped with respect to a circle centered at~$z_K$. The triangles $\{T_i\}_{i=1}^n$ are chosen from the auxiliary triangulation in Lemma~\ref{lem:regAuxTria} as $T_i=T_{E_i}\in\mathcal T_h(K)$, cf.~Figure~\ref{fig:admissibleDecomposition} middle. Obviously, $\{\omega_i\}_{i=1}^n$ is an admissible decomposition, since $\{T_i\}_{i=1}^n$ fulfill the preliminaries.
\begin{figure}[htbp]
 \centering
 \scalebox{0.8}{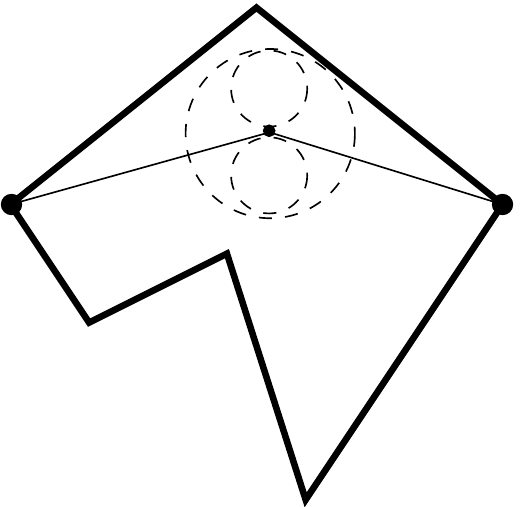}\qquad
 \scalebox{0.8}{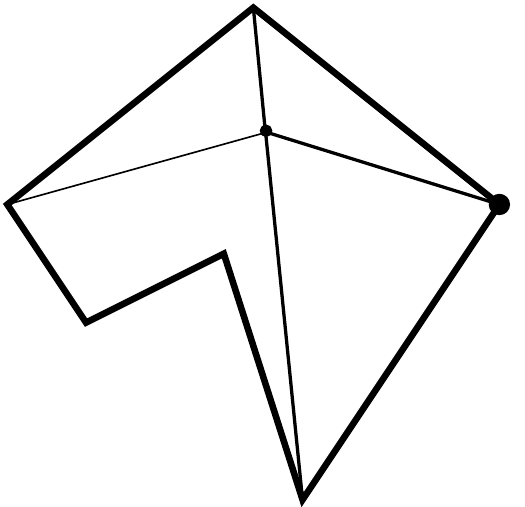}\qquad
 \scalebox{0.8}{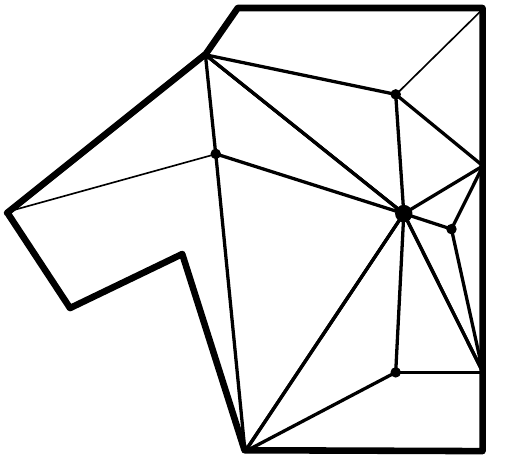}
 \caption{Construction of admissible decomposition for $K$ and $\omega_z$ from Figure~\ref{fig:neighbourhoods}}
 \label{fig:admissibleDecomposition}
\end{figure}

The element~$K$ is star-shaped with respect to a circle of radius~$\rho_K$ and we have split this circle into two circular sectors during the construction of $\omega_i$, $i=1,2$. A small calculation shows that $\omega_i$ is also star-shaped with respect to a circle of radius
\[\rho_{\omega_i}=\frac{\rho_K\sin(\beta/2)}{1+\sin(\beta/2)},\]
which lies inside the mentioned circular sector and fulfills $\rho_K/(1+\sqrt{2})<\rho_{\omega_i}\leq\rho_K/2$, see Figure~\ref{fig:admissibleDecomposition} (left). Thus, the aspect ratio of $\omega_i$ is uniformly bounded, since
\[\frac{h_{\omega_i}}{\rho_{\omega_i}} \leq \frac{(1+\sqrt{2})h_K}{\rho_K} \leq (1+\sqrt{2})\sigma_\mathcal{K}.\]
Furthermore, we observe that $h_{\omega_i}\leq h_K\leq \sigma_\mathcal{K} \rho_K\leq \sigma_\mathcal{K}|\overline{zz_K}|$ and accordingly $h_{\omega_i}\leq \sigma_\mathcal{K}|\overline{z'z_K}|$. Consequently, $\omega_i$, $i=1,2$ is a regular element in the sense of Definition~\ref{def:regularMesh} and, thus, we have already proved that $C_P(\omega_i)\leq c$. Additionally, we obtain by~\eqref{eq:boundAreaOfTE} and by the regularity of the mesh that
\[
 \frac{h_{\omega_i}^2}{|T_i|} 
 \leq \frac{2h_{\omega_i}^2}{h_{E_i} \rho_K} 
 \leq \frac{2h_{K}^2}{h_{E_i} \rho_K} 
 \leq 2c_\mathcal{K}\sigma_\mathcal{K}.
\]
This yields together with $|\omega_z|\leq h_{\omega_z}^2$ and Proposition 2.10 (Decomposition) of~\cite{VeeserVerfuerth2010} that
 \[C_P(\omega_z) \leq \left( 16(n-1)\left(c^2+2c\right) c_\mathcal{K}\sigma_\mathcal{K}\right)^{1/2},\]
and thus, a uniform bound in the case of $\omega_z=K$ and $n=2$.

In the third step, the general case, the patch $\omega_z$ is a union of several elements, see~\eqref{eq:neighbourhoods} and Figure~\ref{fig:neighbourhoods}. In this situation, we proceed the construction for all neighbouring elements of the node~$z$ as in the second step, see Figure~\ref{fig:admissibleDecomposition} (right). Consequently, $n$ is two times the number of neighbouring elements. This number is uniformly bounded according to Lemma~\ref{lem:propertiesMesh}. The resulting decomposition $\{\omega_i\}_{i=1}^n$ is admissible with $\bigcup_{i=1}^n\overline T_i = \overline{\widetilde\omega}_z$ and the estimate of~\cite{VeeserVerfuerth2010} yields $C_P(\omega_z)\leq c$, where $c$ only depends on $\sigma_\mathcal{K}$ and $c_\mathcal{K}$.
\end{proof}

Finally, we formulate approximation properties for the quasi-interpolation operators~$\mathfrak I_h$ and~$\widetilde{\mathfrak I}_h$. The proof is skipped since it is analogous to~\cite{Weisser2011} with the exception, that the generalized Lemmata~\ref{lem:IsoscelesTriangles}, \ref{lem:propertiesMesh}, \ref{lem:step2} and \ref{lem:step1} are applied instead of their counterparts.
\begin{proposition}\label{pro:interpolationProperties}
 Let $\mathcal K_h$ be a regular mesh and let $v\in V$, $E\in\mathcal E_h$ and $K\in\mathcal K_h$.
 Then, it holds
 \[
  \|v-\mathfrak I_hv\|_{0,K} \leq ch_K\vert v\vert_{1,\omega_K},
  \qquad%\mbox{and}\qquad
  \|v-\mathfrak I_hv\|_{0,E} \leq ch_E^{1/2}\vert v\vert_{1,\omega_E},
 \]
 and
 \[
  \|v-\widetilde{\mathfrak I}_hv\|_{0,K} \leq ch_K\vert v\vert_{1,\omega_K},
  \qquad%\mbox{and}\qquad
  \|v-\widetilde{\mathfrak I}_hv\|_{0,E} \leq ch_E^{1/2}\vert v\vert_{1,\omega_E},
 \]
 where the constants $c>0$ only depend on the regularity parameters $\sigma_{\mathcal K}$ and $c_{\mathcal K}$, see Definition~\ref{def:regularMesh}.
\end{proposition}
%\textcolor{red}{
%Direct consequences of this proposition are the estimates
%\[
% \|\mathfrak I_hv\|_{0,K} \leq c\|v\|_{0,\omega_K}
% \quad\mbox{and}\quad
% \|\widetilde{\mathfrak I}_hv\|_{0,K} \leq c\|v\|_{0,\omega_K},
%\]
%which are obtained by the lower triangular inequality.
%}

%% file: fig/decomp_elem.pdf_t
\begin{picture}(0,0)%
\includegraphics{decomp_elem.pdf}%
\end{picture}%
\setlength{\unitlength}{4144sp}%
\begingroup\makeatletter\ifx\SetFigFont\undefined%
\gdef\SetFigFont#1#2#3#4#5{%
  \reset@font\fontsize{#1}{#2pt}%
  \fontfamily{#3}\fontseries{#4}\fontshape{#5}%
  \selectfont}%
\fi\endgroup%
\begin{picture}(2403,2316)(-47,-1444)
\put(1801,434){\makebox(0,0)[lb]{\smash{{\SetFigFont{12}{14.4}{\familydefault}{\mddefault}{\updefault}{\color[rgb]{0,0,0}$E_2$}%
}}}}
\put(2341,-151){\makebox(0,0)[lb]{\smash{{\SetFigFont{12}{14.4}{\familydefault}{\mddefault}{\updefault}{\color[rgb]{0,0,0}$z$}%
}}}}
\put(361,-286){\makebox(0,0)[lb]{\smash{{\SetFigFont{12}{14.4}{\familydefault}{\mddefault}{\updefault}{\color[rgb]{0,0,0}$\omega_1$}%
}}}}
\put(1891,-871){\makebox(0,0)[lb]{\smash{{\SetFigFont{12}{14.4}{\familydefault}{\mddefault}{\updefault}{\color[rgb]{0,0,0}$E_1$}%
}}}}
\put(1576,-1321){\makebox(0,0)[lb]{\smash{{\SetFigFont{12}{14.4}{\familydefault}{\mddefault}{\updefault}{\color[rgb]{0,0,0}$\omega_z=K$}%
}}}}
\put(1098,362){\makebox(0,0)[lb]{\smash{{\SetFigFont{12}{14.4}{\familydefault}{\mddefault}{\updefault}{\color[rgb]{0,0,0}$z_K$}%
}}}}
\put(506,172){\makebox(0,0)[lb]{\smash{{\SetFigFont{12}{14.4}{\familydefault}{\mddefault}{\updefault}{\color[rgb]{0,0,0}$\omega_2$}%
}}}}
\put(-32, 47){\makebox(0,0)[lb]{\smash{{\SetFigFont{12}{14.4}{\familydefault}{\mddefault}{\updefault}{\color[rgb]{0,0,0}$z'$}%
}}}}
\end{picture}%

%% file: fig/decomp_elem_2.pdf_t
\begin{picture}(0,0)%
\includegraphics{decomp_elem_2.pdf}%
\end{picture}%
\setlength{\unitlength}{4144sp}%
\begingroup\makeatletter\ifx\SetFigFont\undefined%
\gdef\SetFigFont#1#2#3#4#5{%
  \reset@font\fontsize{#1}{#2pt}%
  \fontfamily{#3}\fontseries{#4}\fontshape{#5}%
  \selectfont}%
\fi\endgroup%
\begin{picture}(2388,2316)(-32,-1444)
\put(1801,434){\makebox(0,0)[lb]{\smash{{\SetFigFont{12}{14.4}{\familydefault}{\mddefault}{\updefault}{\color[rgb]{0,0,0}$E_2$}%
}}}}
\put(2341,-151){\makebox(0,0)[lb]{\smash{{\SetFigFont{12}{14.4}{\familydefault}{\mddefault}{\updefault}{\color[rgb]{0,0,0}$z$}%
}}}}
\put(361,-286){\makebox(0,0)[lb]{\smash{{\SetFigFont{12}{14.4}{\familydefault}{\mddefault}{\updefault}{\color[rgb]{0,0,0}$\omega_1$}%
}}}}
\put(1891,-871){\makebox(0,0)[lb]{\smash{{\SetFigFont{12}{14.4}{\familydefault}{\mddefault}{\updefault}{\color[rgb]{0,0,0}$E_1$}%
}}}}
\put(506,172){\makebox(0,0)[lb]{\smash{{\SetFigFont{12}{14.4}{\familydefault}{\mddefault}{\updefault}{\color[rgb]{0,0,0}$\omega_2$}%
}}}}
\put(1233,322){\makebox(0,0)[lb]{\smash{{\SetFigFont{12}{14.4}{\familydefault}{\mddefault}{\updefault}{\color[rgb]{0,0,0}$z_K$}%
}}}}
\put(1541,-501){\makebox(0,0)[lb]{\smash{{\SetFigFont{12}{14.4}{\familydefault}{\mddefault}{\updefault}{\color[rgb]{0,0,0}$T_1$}%
}}}}
\put(1521,234){\makebox(0,0)[lb]{\smash{{\SetFigFont{12}{14.4}{\familydefault}{\mddefault}{\updefault}{\color[rgb]{0,0,0}$T_2$}%
}}}}
\put(1576,-1321){\makebox(0,0)[lb]{\smash{{\SetFigFont{12}{14.4}{\familydefault}{\mddefault}{\updefault}{\color[rgb]{0,0,0}$\omega_z=K$}%
}}}}
\end{picture}%

%% file: fig/decomp_patch.pdf_t
\begin{picture}(0,0)%
\includegraphics{decomp_patch.pdf}%
\end{picture}%
\setlength{\unitlength}{4144sp}%
\begingroup\makeatletter\ifx\SetFigFont\undefined%
\gdef\SetFigFont#1#2#3#4#5{%
  \reset@font\fontsize{#1}{#2pt}%
  \fontfamily{#3}\fontseries{#4}\fontshape{#5}%
  \selectfont}%
\fi\endgroup%
\begin{picture}(2348,2092)(5,-1374)
\put(2338,-1244){\makebox(0,0)[lb]{\smash{{\SetFigFont{12}{14.4}{\familydefault}{\mddefault}{\updefault}{\color[rgb]{0,0,0}$\omega_z$}%
}}}}
\put(1621,-376){\makebox(0,0)[lb]{\smash{{\SetFigFont{12}{14.4}{\familydefault}{\mddefault}{\updefault}{\color[rgb]{0,0,0}$z$}%
}}}}
\end{picture}%

%% file: ResidualBasedErrorEstimate.tex
\section{Residual based error estimate}
\label{sec:ResidualErrorEstimate}
In this section, we formulate the main results for the residual based error estimate and prove its reliability and efficiency. 
This a posteriori error estimate bounds the difference of the exact solution and the Galerkin approximation in the energy norm~$\|\cdot\|_b$ associated to the bilinear form, i.e. $\|\cdot\|_b^2 = b(\cdot,\cdot)$. Among others, the estimate contains the jumps of the conormal derivatives over the element edges. Such a jump over an internal edge $E\in\mathcal E_{h,\Omega}$ is defined by
\[ \llbracket u_h\rrbracket_{E,h} = a_K\widetilde{\gamma_1^K u_h} + a_{K^\prime}\widetilde{\gamma_1^{K^\prime} u_h},\]
where $K,K^\prime\in\mathcal K_h$ are the neighbouring elements of $E$ with $E\in\mathcal E(K)\cap\mathcal E(K^\prime)$. 
The element residual is given by
\[R_K = f+a_K\Delta u_h\quad\mbox{for } K\in\mathcal K_h,\]
and the edge residual by
\[
 R_E = \begin{cases}
         0 & \mbox{for } E\in\mathcal E_{h,D},\\
         g_N-a_K\widetilde{\gamma_1^Ku_h} & \mbox{for } E\in\mathcal E_{h,N}\mbox{ with } E\in\mathcal E(K),\\
         -\frac{1}{2}\llbracket u_h\rrbracket_{E,h} & \mbox{for } E\in\mathcal E_{h,\Omega}.
        \end{cases}
\]

\begin{theorem}[Reliability]\label{th:ReliabilityResErrEst}
 Let $\mathcal K_h$ be a regular mesh. Furthermore, let $u\in g_D+V$ and $u_h\in g_D+V_h^k$ be the solutions of~\eqref{eq:VF} and~\eqref{eq:approxDiscVF1}-\eqref{eq:approxDiscVF2}, respectively.
 Then the residual based error estimate is reliable, i.e.
 \[
  \Vert u-u_h\Vert_b \leq c \:\left\{\eta_R^2+\delta_R^2\right\}^{1/2}\quad\mbox{with}\quad
  \eta_R^2 = \sum_{K\in\mathcal K_h}\eta_K^2
  \quad\mbox{and}\quad
  \delta_R^2 = \sum_{K\in\mathcal K_h}\delta_K^2,
 \]
 where the error indicators are defined by
 \[\eta_K^2 = h_K^2\|R_K\|_{0,K}^2+\sum_{E\in\mathcal E(K)}\;h_E\lVert R_E\rVert_{0,E}^2,\]
 and
 \[\delta_K^2 = \|a_K\gamma_1^Ku_h-a_K\widetilde{\gamma_1^Ku_h}\|_{0,\partial K}^2.\]
 The constant $c>0$ only depends on the regularity parameters $\sigma_{\mathcal K}$, $c_{\mathcal K}$, see Definition~\ref{def:regularMesh}, the approximation order~$k$ and on the diffusion coefficient~$a$.
\end{theorem}

The term~$\delta_K$ measures the approximation error in the Neumann traces of the basis functions of $V_h^k$ coming from the boundary element method. 

To state the efficiency, we introduce the notation $\|\cdot\|_{b,\omega}$ for $\omega\subset\Omega$, which means that the energy norm is only computed over the subset~$\omega$. More precisely, it is $\|v\|_{b,\omega}^2 = (a\nabla v,\nabla v)_\omega$ for our model problem.
\begin{theorem}[Efficiency]\label{th:EfficiencyResErrEst}
 Under the assumptions of Theorem~\ref{th:ReliabilityResErrEst}, the residual based error indicator is efficient, i.e.
\begin{eqnarray*}
 \eta_K 
 & \leq & c\,\bigg(\|u-u_h\|_{b,\widetilde\omega_K}^2
                    + h_K^2\|f-\widetilde f\|_{0,\widetilde\omega_K}^2
                    + \sum_{E\in\mathcal E(K)\cap\mathcal E_{h,N}}h_E\|g_N-\widetilde g_N\|_{0,E}^2\\
 &      & \qquad    + \sum_{E\in\mathcal E(K)}\sum_{K'\subset\widetilde\omega_E}h_E\|a_{K'}\gamma_1^{K'}u_h-a_{K'}\widetilde{\gamma_1^{K'}u_h}\|_{0,E}^2\bigg)^{1/2},
\end{eqnarray*}
 where $\widetilde f$ and $\widetilde g_N$ are piecewise polynomial approximations of the data $f$ and $g_N$, respectively.
 The constant $c>0$ only depends on the regularity parameters $\sigma_{\mathcal K}$, $c_{\mathcal K}$, see Definition~\ref{def:regularMesh}, the approximation order~$k$ and on the diffusion coefficient~$a$.
\end{theorem}

The terms involving the data approximation $\|f-\widetilde f\|_{0,\widetilde\omega_K}$ and $\|g_N-\widetilde g_N\|_{0,E}$ are often called data oscillations. They are usually of higher order. Furthermore, the approximation of the Neumann traces by the boundary element method appear in the right hand side. This term is related to~$\delta_K$.

\begin{remark}
 Under certain conditions on the diffusion coefficient it is possible to get the estimates in Theorems~\ref{th:ReliabilityResErrEst} and~\ref{th:EfficiencyResErrEst} robust with respect to~$a$, see, i.e.,~\cite{Petzoldt2002}.
\end{remark}

\subsection{Reliability}
We follow the classical lines in the proof of the reliability, see, i.e.,~\cite{Verfuerth2013}. However, we have to take care on the polygonal elements and the quasi-interpolation operators.
\begin{proof}[Proof (Theorem \ref{th:ReliabilityResErrEst})]
The bilinear form $b(\cdot,\cdot)$ is a scalar product on~$V$ due to its boundedness and coercivity, and thus, $V$ is a Hilbert space together with $b(\cdot,\cdot)$ and $\|\cdot\|_b$. The Riesz representation theorem yields
\[
 \|u-u_h\|_b = \sup_{v\in V\setminus\{0\}}\frac{|\mathcal R(v)|}{\|v\|_b}
 \quad\mbox{with}\quad
 \mathcal R(v) = b(u-u_h,v).
\]
To see the reliability of the residual based error estimate, we use this representation of the error in the energy norm and reformulate and estimate the term $|\mathcal R(v)|$ in the following. Let $v_h\in V_h^1$, after a few manipulations using \eqref{eq:VF} and~\eqref{eq:approxDiscVF1}-\eqref{eq:approxDiscVF2}, we obtain
\begin{eqnarray*}
 \mathcal R(v) 
 & = &   (f,v) + (g_N,v)_{\Gamma_N} 
       - (f,\widetilde v_h) - (g_N,v_h)_{\Gamma_N} \\
 &   & + b_h(u_h,v) - b(u_h,v) 
       + b_h(u_h,v_h) - b_h(u_h,v).
\end{eqnarray*}
The formulas~\eqref{eq:ExactBilinearForm} and~\eqref{eq:ApproxBilinearForm} for $b(\cdot,\cdot)$ and $b_h(\cdot,\cdot)$ lead to
\begin{eqnarray*}
 \mathcal R(v) 
 & = & \sum_{K\in\mathcal K_h} \Big\{ (f,v-\widetilde v_h)_K + (g_N,v-v_h)_{\partial K\cap\Gamma_N}\\[-1.5ex]
 &   & \quad\qquad                   - a_K\big\{(\gamma_1^Ku_h-\widetilde{\gamma_1^Ku_h},v)_{\partial K} - (\Delta u_h, v-\widetilde v)_K\big\}\\
 &   & \quad\qquad                   - a_K\big\{(\widetilde{\gamma_1^Ku_h},v-v_h)_{\partial K} - (\Delta u_h, \widetilde v-\widetilde v_h)_K\big\}\Big\}.
\end{eqnarray*}
After some rearrangements of the sums, we obtain
\begin{eqnarray}
 \mathcal R(v) 
 & = & \sum_{K\in\mathcal K_h} \!\Big\{ (f+a_K\Delta u_h,v-\widetilde v_h)_K + (g_N,v-v_h)_{\partial K\cap\Gamma_N}\nonumber\\[-1ex]
 &   & \quad\qquad                  - (a_K\widetilde{\gamma_1^Ku_h},v-v_h)_{\partial K}- a_K(\gamma_1^Ku_h-\widetilde{\gamma_1^Ku_h},v)_{\partial K}\Big\}\label{eq:Residual}\\
 & = & \sum_{K\in\mathcal K_h} \!\Big\{ (R_K,v-\widetilde v_h)_K + \!\!\sum_{E\in\mathcal E(K)}\!\!(R_E,v-v_h)_{E} - a_K(\gamma_1^Ku_h-\widetilde{\gamma_1^Ku_h},v)_{\partial K}\Big\}.\nonumber
\end{eqnarray}
The Cauchy-Schwarz and the triangular inequality yield
%\begin{eqnarray}
% \lefteqn{|(R_K,v-\widetilde v_h)_K + \sum_{E\in\mathcal E(K)}(R_E,v-v_h)_{E} - a_K(\gamma_1^Ku_h-\widetilde{\gamma_1^Ku_h},v)_{\partial K}|}\nonumber\\
% & \leq &   \|R_K\|_{0,K} \|v-\widetilde v_h\|_{0,K}
%          + \sum_{E\in\mathcal E(K)}\|R_E\|_{0,E}\|v-v_h\|_{0,E} 
%          + \delta_K\|v\|_{0,\partial K}.\label{eq:EstTermResidual}
%\end{eqnarray}
\begin{equation}\label{eq:EstTermResidual}
 \mathcal R(v) 
  \leq \sum_{K\in\mathcal K_h} \Big\{
           \|R_K\|_{0,K} \|v-\widetilde v_h\|_{0,K}
           + \sum_{E\in\mathcal E(K)}\|R_E\|_{0,E}\|v-v_h\|_{0,E} 
           + \delta_K\|v\|_{0,\partial K}
       \Big\}.
\end{equation}
We have introduced the notation $\widetilde v$ for the polynomial approximation of $v$. Consequently, we obtain for the approximation~$\widetilde v_h$ of $v_h$ according to Lemma~(4.3.8) in~\cite{BrennerScott2002}
\[
 \|v-\widetilde v_h\|_{0,K} 
 \leq \|v-v_h\|_{0,K} + \|v_h-\widetilde v_h\|_{0,K}
 \leq \|v-v_h\|_{0,K} + ch_K|v_h|_{1,K}
\]
with a constant~$c$ only depending on~$\sigma_\mathcal{K}$ and~$k$.
Let $\mathfrak I_{h,C}$ be the usual Cl\'ement interpolation operator over the auxiliary triangulation~$\mathcal T_h(\Omega)$, which maps into the space of piecewise linear and globally continuous functions, see~\cite{Clement1975}. Due to the construction of $\widetilde{\mathfrak I}_h$, it is $\widetilde{\mathfrak I}_hv\big|_{\partial K}=\mathfrak I_{h,C}v\big|_{\partial K}$ for $K\in\mathcal K_h$. Since $\widetilde{\mathfrak I}_hv$ is (weakly) harmonic on~$K$, it minimizes the energy such that
\[
  |\widetilde{\mathfrak I}_hv|_{1,K} 
  = \min\{|w|_{1,K}:w\in H^1(K),w=\widetilde{\mathfrak I}_hv\mbox{ on } \partial K\} 
  \leq |\mathfrak I_{h,C}v|_{1,K}.
\]
We choose $v_h=\widetilde{\mathfrak I}_hv$ and obtain
\begin{eqnarray*}
 \|v-\widetilde v_h\|_{0,K} 
 & \leq & \|v-\widetilde{\mathfrak I}_hv\|_{0,K} + ch_K|\mathfrak I_{h,C}v|_{1,K}\\
 & \leq & \|v-\widetilde{\mathfrak I}_hv\|_{0,K} + ch_K(|v|_{1,K} + |v-\mathfrak I_{h,C}v|_{1,K})\\
 & \leq & \|v-\widetilde{\mathfrak I}_hv\|_{0,K} + ch_K\bigg(|v|_{1,K} + \Big(\sum_{z\in\mathcal N(K)}|v|_{1,\widetilde\omega_z}^2 + |v|_{1,K}^2\Big)^{1/2}\bigg),
\end{eqnarray*}
where an interpolation error estimate for the Cl\'ement operator has been applied, see~\cite{Clement1975}. Proposition~\ref{pro:interpolationProperties} and $|\mathcal N(K)|<c$, according to Lemma~\ref{lem:propertiesMesh}, yield
\[\|v-\widetilde v_h\|_{0,K} \leq ch_K|v|_{1,\omega_K}.\]
Estimating $\|v-v_h\|_{0,E}$ with Proposition~\ref{pro:interpolationProperties} gives
\begin{eqnarray*}
 \sum_{E\in\mathcal E(K)}\|R_E\|_{0,E}\|v-\widetilde{\mathfrak I}_hv\|_{0,E} 
 & \leq & \sum_{E\in\mathcal E(K)}ch_E^{1/2}\|R_E\|_{0,E}|v|_{1,\omega_E} \\
 & \leq & c|v|_{1,\omega_K}\Big(\sum_{E\in\mathcal E(K)}h_E\|R_E\|_{0,E}^2\Big)^{1/2},
\end{eqnarray*}
where we have used again Cauchy-Schwarz and Lemma~\ref{lem:propertiesMesh}.
We combine the previous estimates and apply the trace inequality in such a way that the residual in~\eqref{eq:EstTermResidual} is bounded by
\begin{eqnarray*}
 \mathcal R(v) 
  & \leq & c\sum_{K\in\mathcal K_h} \Big\{
              h_K\|R_K\|_{0,K}|v|_{1,\omega_K}
              + \Big(\sum_{E\in\mathcal E(K)}h_E\|R_E\|_{0,E}^2\Big)^{1/2}|v|_{1,\omega_K}
              + \delta_K\|v\|_{1,K}
           \Big\}\\
 & \leq & c\sum_{K\in\mathcal K_h}\Big\{\eta_K\:|v|_{1,\omega_K} + \delta_K\:\|v\|_{1,K}\Big\}
% & \leq & c\sum_{K\in\mathcal K_h}\Big\{\eta_K^2+ \delta_K^2\Big\}^{1/2}\|v\|_{1,\omega_K} \\
   \leq   c\:\eta_R\:\|v\|_{1,\Omega}.
%  & \leq & c\sum_{K\in\mathcal K_h} \Big\{
%              h_K\|R_K\|_{0,K}
%              + \Big(\sum_{E\in\mathcal E(K)}h_E\|R_E\|_{0,E}^2\Big)^{1/2}
%              + \delta_K
%           \Big\}\|v\|_{1,\omega_K}\\
%  & \leq & c\,\Big\{\sum_{K\in\mathcal K_h}
%              h_K^2\|R_K\|_{0,K}^2
%              + \sum_{E\in\mathcal E(K)}h_E\|R_E\|_{0,E}^2
%              + \delta_K^2
%           \Big\}^{1/2}
%           \Big\{\sum_{K\in\mathcal K_h}\|v\|_{1,\omega_K}^2 \Big\}^{1/2}\\
%  & \leq & c\,\Big\{\sum_{K\in\mathcal K_h}
%              \eta_K^2
%              + \delta_K^2
%           \Big\}^{1/2}
%           \|v\|_{1,\Omega}
\end{eqnarray*}
The last estimate is valid since each element is covered by a finite number of patches, see Lemma~\ref{lem:propertiesMesh}.
The norm~$\|\cdot\|_{1,\Omega}$ and the semi-norm~$|\cdot|_{1,\Omega}$ are equivalent on~$V$ and $\sqrt{a/a_\mathrm{min}}>1$. Therefore, we conclude
\[\vert\mathcal R(v)\vert \leq \frac{c}{\sqrt{a_\mathrm{min}}}\:\eta_R\:\Vert v\Vert_b,\]
which finishes the proof.
\end{proof}

\subsection{Efficiency}
We adapt the bubble function technique to polygonal meshes. Therefore, let $\phi_T$ and $\phi_E$ be the usual polynomial bubble functions over the auxiliary triangulation~$\mathcal T_h(\Omega)$, see~\cite{AinsworthOden2000,Verfuerth2013}. Here, $\phi_T$ is a quadratic polynomial over the triangle $T\in\mathcal T_h(\Omega)$, which vanishes on $\Omega\setminus T$ and in particular on~$\partial T$. The edge bubble $\phi_E$ is a piecewise quadratic polynomial over the adjacent triangles in~$\mathcal T_h(\Omega)$, sharing the common edge~$E$, and it vanishes elsewhere. We define the new bubble functions over the polygonal mesh as
\[
 \varphi_K = \sum_{T\in\mathcal T_h(K)} \phi_T
 \qquad\mbox{and}\qquad
 \varphi_E = \phi_E
\]
for $K\in\mathcal K_h$ and $E\in\mathcal E_h$.
\begin{lemma}\label{lem:BubbleFunc}
 Let $K\in\mathcal K_h$ and $E\in\mathcal E(K)$. The bubble functions satisfy
 \[
  \begin{aligned}
   \supp \varphi_K &= K, & 0\leq&\varphi_K\leq1,\\
   \supp \varphi_E &\subset \widetilde\omega_E, & 0\leq&\varphi_E\leq1,
  \end{aligned}
 \]
 and fulfill for $p\in\mathcal P^k(K)$ the estimates
 \begin{align*}
  \|p\|_{0,K}^2 &\leq c\, (\varphi_Kp,p)_K, &
  |\varphi_Kp|_{1,K} &\leq ch_K^{-1}\, \|p\|_{0,K},\\
  \|p\|_{0,E}^2 &\leq c\, (\varphi_Ep, p)_E, &
  |\varphi_Ep|_{1,K} &\leq ch_E^{-1/2}\,\|p\|_{0,E},\\
  && \|\varphi_Ep\|_{0,K} &\leq ch_E^{1/2}\,\|p\|_{0,E}.
 \end{align*}
 The constants $c>0$ only depend on the regularity parameters $\sigma_{\mathcal K}$, $c_{\mathcal K}$ and on the approximation order~$k$.
\end{lemma}
\begin{proof}
 Similar estimates are valid for $\phi_T$ and $\phi_E$ on triangular meshes, see~\cite{AinsworthOden2000,Verfuerth2013}. By the use of Cauchy-Schwarz inequality and the properties of the auxiliary triangulation~$\mathcal T_h(\Omega)$ the estimates translate to the new bubble functions. The details of the proof are omitted. 
\end{proof}

With these ingredients the proof of Theorem~\ref{th:EfficiencyResErrEst} can be addressed. The arguments follow the line of~\cite{AinsworthOden2000}.
\begin{proof}[Proof (Theorem~\ref{th:EfficiencyResErrEst})]
Let $\widetilde R_K\in\mathcal P^k(K)$ be a polynomial approximation of the element residual~$R_K$ for $K\in\mathcal K_h$. For $v=\varphi_K\widetilde R_K\in H^1_0(K)$ and $v_h=0$ equation~\eqref{eq:Residual} yields
\[
  b(u-u_h,\varphi_K\widetilde R_K) = \mathcal R(\varphi_K\widetilde R_K) = (R_K,\varphi_K\widetilde R_K)_K.
\]
Lemma~\ref{lem:BubbleFunc} gives
\begin{eqnarray*}
 \|\widetilde R_K\|_{0,K}^2 
 & \leq & c\,(\varphi_K\widetilde R_K,\widetilde R_K)_K\\
 &  =   & c\left((\varphi_K\widetilde R_K,\widetilde R_K-R_K)_K + (\varphi_K\widetilde R_K,R_K)_K\right)\\
 & \leq & c\left(\|\widetilde R_K\|_{0,K}\|\widetilde R_K-R_K\|_{0,K} + b(u-u_h,\varphi_K\widetilde R_K)\right),
\end{eqnarray*}
and furthermore
\[
  b(u-u_h,\varphi_K\widetilde R_K)
  \leq c\, |u-u_h|_{1,K}|\varphi_K\widetilde R_K|_{1,K}
  \leq ch_K^{-1}\, \|u-u_h\|_{b,K}\|\widetilde R_K\|_{0,K}.
\]
We thus get
\[\|\widetilde R_K\|_{0,K} \leq c\left(h_K^{-1}\|u-u_h\|_{b,K}+\|\widetilde R_K-R_K\|_{0,K}\right),\]
and by the lower triangular inequality
\[\|R_K\|_{0,K} \leq c\left(h_K^{-1}\|u-u_h\|_{b,K}+\|\widetilde R_K-R_K\|_{0,K}\right).\]

Next, we consider the edge residual. Let $\widetilde R_E\in\mathcal P^k(E)$ be an approximation of $R_E$, with $E\in\mathcal E_{h,\Omega}$.
% with $E\in\mathcal E(K_1)\cap\mathcal E(K_2)$, $K_1\neq K_2$ and corresponding triangles $T_{E,1}\in\mathcal T_h(K_1)$ and $T_{E,2}\in\mathcal T_h(K_2)$. 
The case $E\in\mathcal E_{h,N}$ is treated analogously. For $v=\varphi_E\widetilde R_E\in H^1_0(\widetilde\omega_E)$ and $v_h=0$ equation~\eqref{eq:Residual} yields this time
\begin{eqnarray*}
 \lefteqn{b(u-u_h,\varphi_E\widetilde R_E) = \mathcal R(\varphi_E\widetilde R_E)}\\
 &\quad = & \sum_{K\subset\widetilde\omega_E}\left\{(R_K,\varphi_E\widetilde R_E)_K + (R_E,\varphi_E\widetilde R_E)_E - a_K(\gamma_1^Ku_h-\widetilde{\gamma_1^Ku_h},\varphi_E\widetilde R_E)_E\right\}.
\end{eqnarray*}
Applying Lemma~\ref{lem:BubbleFunc} and the previous formula leads to
\begin{eqnarray*}
 \|\widetilde R_E\|_{0,E}^2
 & \leq & c\,(\varphi_E\widetilde R_E, \widetilde R_E)_E\\
 &  =   & c\left((\varphi_E\widetilde R_E, \widetilde R_E-R_E)_E + (\varphi_E\widetilde R_E, R_E)_E\right)\\
 & \leq & c\left(\|\widetilde R_E\|_{0,E}\|\widetilde R_E-R_E\|_{0,E} + (\varphi_E\widetilde R_E, R_E)_E\right),
\end{eqnarray*}
and
\begin{eqnarray*}
 \lefteqn{|(\varphi_E\widetilde R_E, R_E)_E|}\\
 & =    & \tfrac12\, \bigg|b(u-u_h,\varphi_E\widetilde R_E) - \sum_{K\subset\widetilde\omega_E}\Big\{(R_K,\varphi_E\widetilde R_E)_K 
                 - a_K(\gamma_1^Ku_h-\widetilde{\gamma_1^Ku_h},\varphi_E\widetilde R_E)_E\Big\}\bigg|\\
 & \leq & c\, \bigg(|u-u_h|_{1,\widetilde\omega_E}|\varphi_E\widetilde R_E|_{1,\widetilde\omega_E} \\
 &      &\qquad + \sum_{K\subset\widetilde\omega_E}\Big\{\|R_K\|_{0,K}\|\varphi_E\widetilde R_E\|_{0,K}
          + a_K\|\gamma_1^Ku_h-\widetilde{\gamma_1^Ku_h}\|_{0,E}\|\widetilde R_E\|_{0,E}\Big\}\bigg)\\
 & \leq & c\, \bigg(h_E^{-1/2}\|u-u_h\|_{b,\widetilde\omega_E} + \sum_{K\subset\widetilde\omega_E}\Big\{h_E^{1/2}\|R_K\|_{0,K} 
          + a_K\|\gamma_1^Ku_h-\widetilde{\gamma_1^Ku_h}\|_{0,E}\Big\}\bigg)\,\|\widetilde R_E\|_{0,E}.
\end{eqnarray*}
Therefore, it is
\begin{eqnarray*}
 \|\widetilde R_E\|_{0,E} 
 & \leq & c\bigg(
            h_E^{-1/2}\|u-u_h\|_{b,\widetilde\omega_E} 
          + \sum_{K\subset\widetilde\omega_E}h_E^{1/2}\|R_K\|_{0,K}\\ 
 &&\qquad + \|\widetilde R_E-R_E\|_{0,E}
          + \sum_{K\subset\widetilde\omega_E}a_K\|\gamma_1^Ku_h-\widetilde{\gamma_1^Ku_h}\|_{0,E}
       \bigg).
\end{eqnarray*}
By the lower triangular inequality, $h_K^{-1}\leq h_E^{-1}$ and the previous estimate for $\|R_K\|_{0,K}$ we obtain
\begin{eqnarray*}
 \|R_E\|_{0,E} 
 & \leq & c\bigg(
            h_E^{-1/2}\|u-u_h\|_{b,\widetilde\omega_E} 
          + \sum_{K\subset\widetilde\omega_E}h_E^{1/2}\|\widetilde R_K-R_K\|_{0,K}\\ 
 &&\qquad + \|\widetilde R_E-R_E\|_{0,E}
          + \sum_{K\subset\widetilde\omega_E}a_K\|\gamma_1^Ku_h-\widetilde{\gamma_1^Ku_h}\|_{0,E}
       \bigg).
\end{eqnarray*}
Let $\widetilde f$ and $\widetilde g_N$ be piecewise polynomial approximations of $f$ and $g_N$, respectively. We choose $\widetilde R_K=\widetilde f+a_K\Delta u_h$ for $K\in\mathcal K_h$, $\widetilde R_E =\widetilde g_N-a_K\widetilde{\gamma_1^Ku_h}$ for $E\in\mathcal E(K)\cap\mathcal E_{h,N}$ and $\widetilde R_E = R_E$ for $E\in\mathcal E_h\setminus \mathcal E_{h,N}$. Consequently, we have $\widetilde R_K\in\mathcal P^k(K)$ and $\widetilde R_E\in\mathcal P^k(E)$. Finally, the estimates for $\|R_K\|_{0,K}$ and $\|R_E\|_{0,E}$ yield after some applications of the Cauchy-Schwarz inequality and Lemma~\ref{lem:propertiesMesh}
\begin{eqnarray*}
 \eta_R^2 
 & \leq & c\,\bigg(\|u-u_h\|_{b,\widetilde\omega_K}^2
                    + h_K^2\sum_{K'\subset\widetilde\omega_K}\|\widetilde R_{K'}-R_{K'}\|_{0,K'}^2\\
 &      & \qquad    + \sum_{E\in\mathcal E(K)}h_E\Big\{\|\widetilde R_E-R_E\|_{0,E}^2
                    + \sum_{K'\subset\widetilde\omega_E}\|a_{K'}\gamma_1^{K'}u_h-a_{K'}\widetilde{\gamma_1^{K'}u_h}\|_{0,E}^2\Big\}\bigg)\\
 & \leq & c\,\bigg(\|u-u_h\|_{b,\widetilde\omega_K}^2
                    + h_K^2\|f-\widetilde f\|_{0,\widetilde\omega K}^2
                    + \sum_{E\in\mathcal E(K)\cap\mathcal E_{h,N}}h_E\|g_N-\widetilde g_N\|_{0,E}^2\\
 &      & \qquad    + \sum_{E\in\mathcal E(K)}\sum_{K'\subset\widetilde\omega_E}h_E\|a_{K'}\gamma_1^{K'}u_h-a_{K'}\widetilde{\gamma_1^{K'}u_h}\|_{0,E}^2\bigg).
\end{eqnarray*}
\end{proof}

\subsection{Application on uniform meshes}
The residual based error estimate can be used as stopping criteria to check if the desired accuracy is reached in a simulation on a sequence of meshes. However, it is well known that residual based estimators overestimate the true error a lot. But, because of the equivalence of the norms $\|\cdot\|_{1,\Omega}$ and $\|\cdot\|_b$ on $V$, we can still use $\eta_R$ to verify numerically the convergence rates for uniform mesh refinement when $h\to0$.

\subsection{Application in adaptive FEM}
\label{subsec:ApplicationAFEM}
The classical adaptive finite element strategy proceeds in the following steps
\[ \mathit{SOLVE} \to \mathit{ESTIMATE} \to \mathit{MARK} \to \mathit{REFINE} \to \mathit{SOLVE} \to \cdots. \]
Sometimes an additional $\mathit{COARSENING}$ step is introduced. We have all ingredient to formulate an adaptive BEM-based FEM on polygonal meshes.

In the $\mathit{SOLVE}$ step, we approximate the solution of the boundary value problem on a given polygonal mesh in the high order trial space~$V_h^k$ with the help of the BEM-based FEM.

The $\mathit{ESTIMATE}$ part is devoted to the computation of local error indicators which are used to gauge the approximation accuracy over each element. Here, we use the term $\eta_K$ from the residual based error estimate.

In $\mathit{MARK}$, we choose some elements according to their indicator $\eta_K$ for refinement. Several marking strategies are possible, we implemented the D\"orflers marking, see~\cite{Doerfler1996}.

Finally in $\mathit{REFINE}$, the marked elements are refined, and thus, a problem adapted mesh is generated. In the implementation each marked element is split into two new ones, as proposed in~\cite{Weisser2011}, and the regularity conditions of the mesh are checked. At this point, we like to stress that the polygonal meshes are beneficial in this context, since hanging nodes are naturally included. Consequently, there is no need for an additional effort avoiding hanging nodes or treating them as conditional degrees of freedom. Also a $\mathit{COARSENING}$ of the mesh is very easy. This can be achieved by simply gluing elements together.

%% file: NumericalExperiments.tex
\section{Numerical experiments}
\label{sec:NumericalExperiments} 
In the following we present numerical examples on uniform and adaptive refined meshes, see Figure~\ref{fig:Meshes}. For the convergence analysis, we consider the error with respect to the mesh size $h=\max\{h_K:K\in\mathcal K_h\}$ for uniform refinement. This makes no sense for adaptive strategies. Since the relation
\[\mathrm{DoF} = O(h^{-2})\]
holds for the number of degrees of freedom (DoF) on uniform meshes, we study the convergence of the adaptive BEM-based FEM with respect to them. 

\begin{figure}[htbp]
 \centering
 \includegraphics[trim=3.0cm 6.0cm 3.5cm 6.4cm, width=0.32\textwidth, clip]{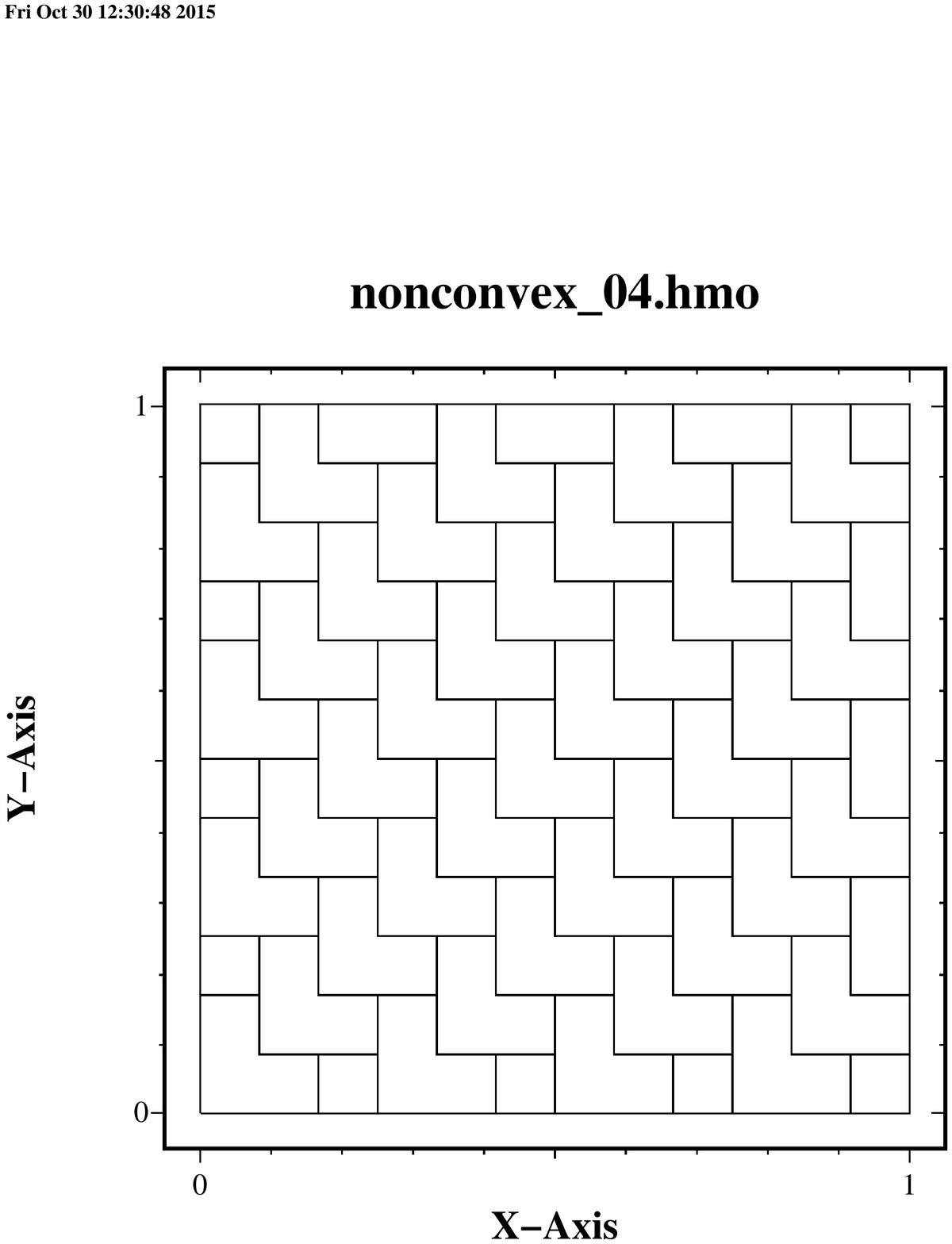}
 \includegraphics[trim=3.0cm 6.0cm 3.5cm 6.4cm, width=0.32\textwidth, clip]{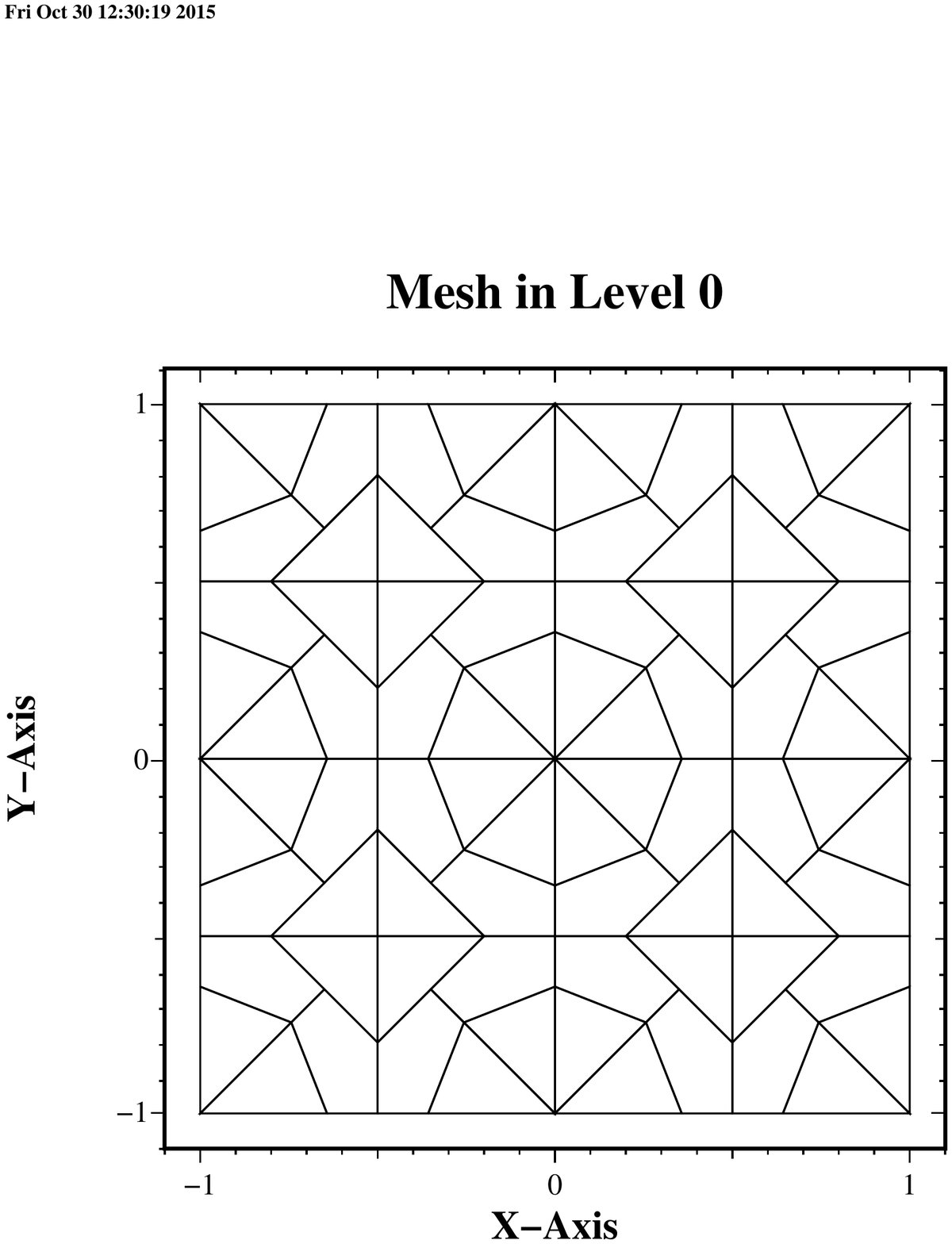}
 \includegraphics[trim=3.0cm 6.0cm 3.5cm 6.4cm, width=0.32\textwidth, clip]{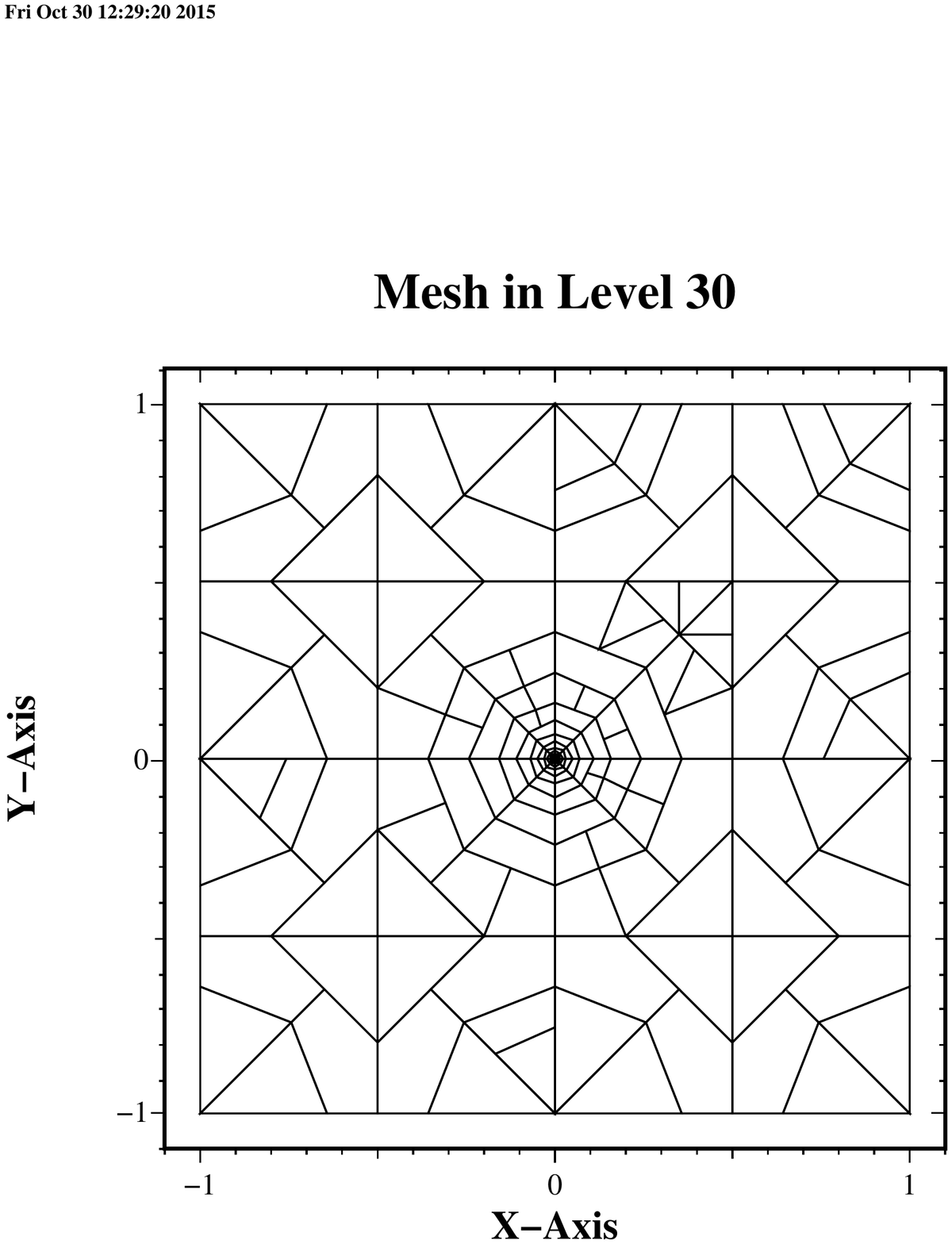}
 \vspace*{-1.5ex}
 \caption{Mesh with L-shaped elements for uniform refinement (left), initial mesh for adaptive refinement (middle), adaptive refined mesh after 30 steps for $k=2$ (right)}
 \label{fig:Meshes}
\end{figure}

\subsection{Uniform refinement strategy}
Consider the Dirichlet boundary value problem
\[
 \begin{aligned}
  -\Delta u &= f &&\mbox{in } \Omega=(0,1)^2,\\
  u &= 0 &&\mbox{on } \Gamma,
 \end{aligned}
\]
where $f\in L_2(\Omega)$ is chosen in such a way that $u(x)=\sin(\pi x_1)\sin(\pi x_2)$ for $x\in\Omega$ is the exact solution. The solution in smooth, and thus, we expect optimal rates of convergence for uniform mesh refinement. The problem is treated with the BEM-based FEM for different approximation orders~$k=1,2,3$ on a sequence of meshes with L-shaped elements of decreasing diameter, see Figure~\ref{fig:Meshes} left. In Figure~\ref{fig:ConvergenceUniform}, we give the convergence graphs in logarithmic scale for the value $\eta_R/|u|_{1,\Omega}$, which behaves like the relative $H^1$-error, and the relative $L_2$-error with respect to the mesh size~$h$. The example confirms the theoretical convergence rates stated in Subsection~\ref{subsec:DiscVF}. 
\begin{figure}[htbp]
 \centering
 \scalebox{0.9}{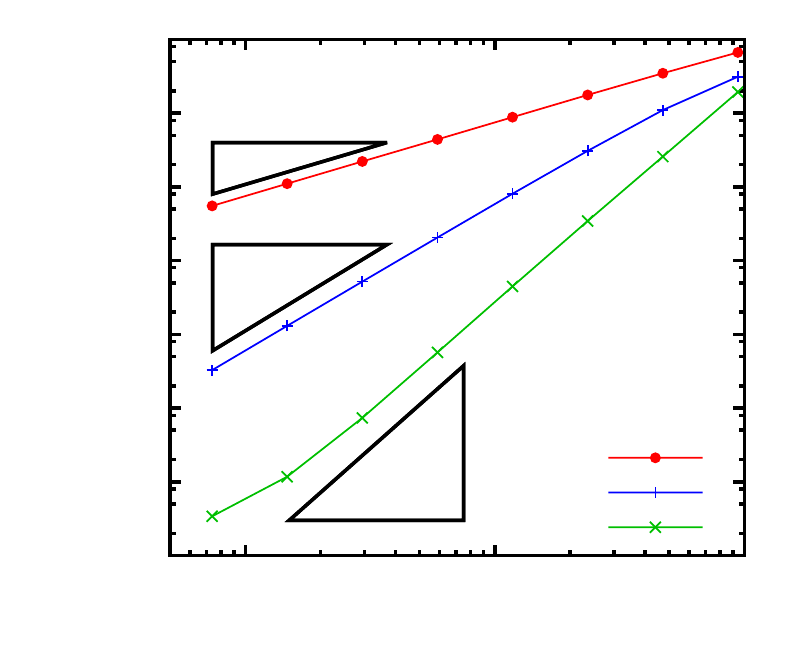}
 \scalebox{0.9}{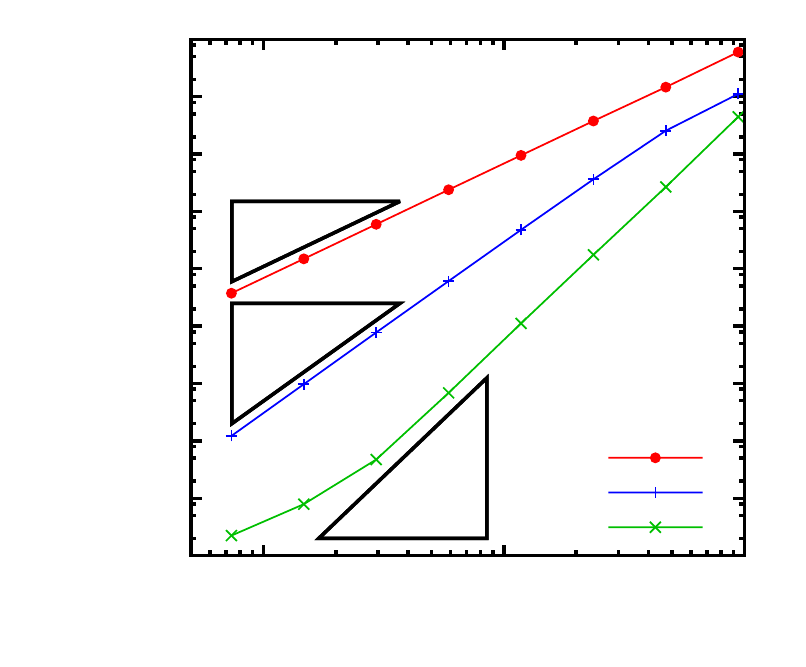}
 \vspace*{-1.5ex}
 \caption{Convergence graph for sequence of uniform meshes and $V_h^k$, $k=1,2,3$, $\eta_R/|u|_{1,\Omega}$ and relative $L_2$-error are given with respect to~$h$ in logarithmic scale}
 \label{fig:ConvergenceUniform}
\end{figure}

\subsection{Adaptive refinement strategy}
Let $\Omega = (-1,1)\times(-1,1)\subset\mathbb{R}^2$ be split into two domains, $\Omega_1 = \Omega\setminus\overline\Omega_2$ and $\Omega_2 = (0,1)\times(0,1)$. Consider the boundary value problem
\[
 \begin{aligned}
  -\div\left( a\nabla u\right) & = 0 && \mbox{in } \Omega, \\
  u & = g && \mbox{on } \Gamma,
 \end{aligned}
\]
where the coefficient $a$ is given by
\[ a = \left\{\begin{array}{rl}
                   1 & \mbox{in } \Omega_1, \\
                   100 & \mbox{in } \Omega_2.
                \end{array}\right.\]
Using polar coordinates $(r,\varphi)$, we choose the boundary data as restriction of the global function
\[ g(x) = r^\lambda\left\{\begin{array}{rl}
                            \cos(\lambda(\varphi-\pi/4)) & \mbox{for } x\in\mathbb R^2_+, \\
                            \beta\cos(\lambda(\pi-|\varphi-\pi/4|)) & \mbox{else,}
                          \end{array}\right.\]
with
\[ \lambda =  \frac{4}{\pi}\arctan\left(\sqrt{\frac{103}{301}}\right) \quad\mbox{and}\quad
   \beta = -100\frac{\sin\left(\lambda\frac{\pi}{4}\right)}{\sin\left(\lambda\frac{3\pi}{4}\right)}. \]
This problem is constructed in such a way that $u=g$ is the exact solution in $\Omega$. Due to the ratio of the jumping coefficient it is $u\not\in H^2(\Omega)$ with a singularity in the origin of the coordinate system. Consequently, uniform mesh refinement does not yield optimal rates of convergence. Since $f=0$, it suffices to approximate the solution in $V_{h,1}^k$ with the variational formulation~\eqref{eq:approxDiscVF1}. We implemented an adaptive strategy according to Subsection~\ref{subsec:ApplicationAFEM}, where the introduced residual based error indicator is utilized in the $\mathit{ESTIMATE}$ step. Starting from an initial polygonal mesh, see Figure~\ref{fig:Meshes} middle, the adaptive BEM-based FEM produces a sequence of locally refined meshes. The approach detects the singularity in the origin of the coordinate system and polygonal elements appear naturally during the local refinement, see Figure~\ref{fig:Meshes} right. In Figure~\ref{fig:ConvergenceAdaptive}, the energy error $\|u-u_h\|_b$ as well as the error estimator~$\eta_R$ are plotted with respect to the number of degrees of freedom in logarithmic scale. As expected by the theory the residual based error estimate represents the behaviour of the energy error very well. Furthermore, the adaptive approach yields optimal rates of convergence in the presence of a singularity.

\begin{figure}[htbp]
 \centering
 \scalebox{0.9}{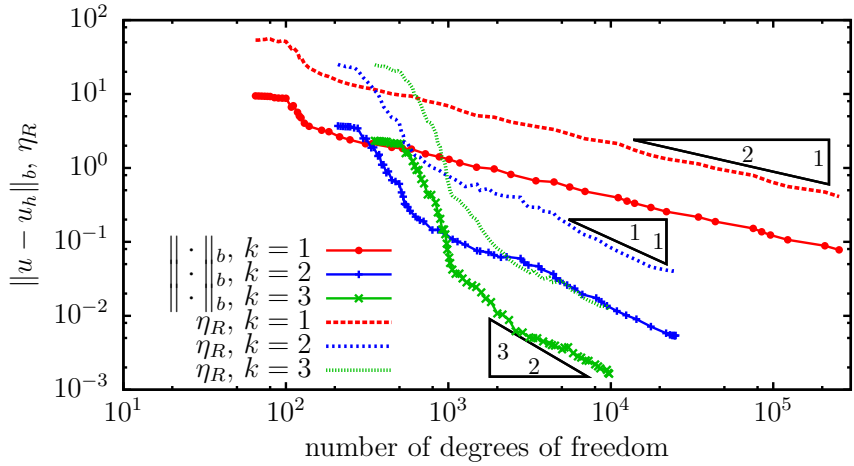}
 \vspace*{-1.5ex}
 \caption{Convergence graph for adaptive BEM-based FEM with $V_{h,1}^k$, $k=1,2,3$, the energy error and the residual based error estimator are given with respect to the number of degrees of freedom in logarithmic sale}
 \label{fig:ConvergenceAdaptive}
\end{figure}

%% file: fig/SinSin_uniform_Eta.pdf_t
% GNUPLOT: LaTeX picture with Postscript
\begingroup
  \makeatletter
  \providecommand\color[2][]{%
    \GenericError{(gnuplot) \space\space\space\@spaces}{%
      Package color not loaded in conjunction with
      terminal option `colourtext'%
    }{See the gnuplot documentation for explanation.%
    }{Either use 'blacktext' in gnuplot or load the package
      color.sty in LaTeX.}%
    \renewcommand\color[2][]{}%
  }%
  \providecommand\includegraphics[2][]{%
    \GenericError{(gnuplot) \space\space\space\@spaces}{%
      Package graphicx or graphics not loaded%
    }{See the gnuplot documentation for explanation.%
    }{The gnuplot epslatex terminal needs graphicx.sty or graphics.sty.}%
    \renewcommand\includegraphics[2][]{}%
  }%
  \providecommand\rotatebox[2]{#2}%
  \@ifundefined{ifGPcolor}{%
    \newif\ifGPcolor
    \GPcolortrue
  }{}%
  \@ifundefined{ifGPblacktext}{%
    \newif\ifGPblacktext
    \GPblacktexttrue
  }{}%
  % define a \g@addto@macro without @ in the name:
  \let\gplgaddtomacro\g@addto@macro
  % define empty templates for all commands taking text:
  \gdef\gplbacktext{}%
  \gdef\gplfronttext{}%
  \makeatother
  \ifGPblacktext
    % no textcolor at all
    \def\colorrgb#1{}%
    \def\colorgray#1{}%
  \else
    % gray or color?
    \ifGPcolor
      \def\colorrgb#1{\color[rgb]{#1}}%
      \def\colorgray#1{\color[gray]{#1}}%
      \expandafter\def\csname LTw\endcsname{\color{white}}%
      \expandafter\def\csname LTb\endcsname{\color{black}}%
      \expandafter\def\csname LTa\endcsname{\color{black}}%
      \expandafter\def\csname LT0\endcsname{\color[rgb]{1,0,0}}%
      \expandafter\def\csname LT1\endcsname{\color[rgb]{0,1,0}}%
      \expandafter\def\csname LT2\endcsname{\color[rgb]{0,0,1}}%
      \expandafter\def\csname LT3\endcsname{\color[rgb]{1,0,1}}%
      \expandafter\def\csname LT4\endcsname{\color[rgb]{0,1,1}}%
      \expandafter\def\csname LT5\endcsname{\color[rgb]{1,1,0}}%
      \expandafter\def\csname LT6\endcsname{\color[rgb]{0,0,0}}%
      \expandafter\def\csname LT7\endcsname{\color[rgb]{1,0.3,0}}%
      \expandafter\def\csname LT8\endcsname{\color[rgb]{0.5,0.5,0.5}}%
    \else
      % gray
      \def\colorrgb#1{\color{black}}%
      \def\colorgray#1{\color[gray]{#1}}%
      \expandafter\def\csname LTw\endcsname{\color{white}}%
      \expandafter\def\csname LTb\endcsname{\color{black}}%
      \expandafter\def\csname LTa\endcsname{\color{black}}%
      \expandafter\def\csname LT0\endcsname{\color{black}}%
      \expandafter\def\csname LT1\endcsname{\color{black}}%
      \expandafter\def\csname LT2\endcsname{\color{black}}%
      \expandafter\def\csname LT3\endcsname{\color{black}}%
      \expandafter\def\csname LT4\endcsname{\color{black}}%
      \expandafter\def\csname LT5\endcsname{\color{black}}%
      \expandafter\def\csname LT6\endcsname{\color{black}}%
      \expandafter\def\csname LT7\endcsname{\color{black}}%
      \expandafter\def\csname LT8\endcsname{\color{black}}%
    \fi
  \fi
  \setlength{\unitlength}{0.0500bp}%
  \begin{picture}(4648.00,3854.00)%
    \gplgaddtomacro\gplbacktext{%
      \csname LTb\endcsname%
      \put(860,640){\makebox(0,0)[r]{\strut{}$10^{-8}$}}%
      \put(860,1065){\makebox(0,0)[r]{\strut{}$10^{-7}$}}%
      \put(860,1489){\makebox(0,0)[r]{\strut{}$10^{-6}$}}%
      \put(860,1914){\makebox(0,0)[r]{\strut{}$10^{-5}$}}%
      \put(860,2339){\makebox(0,0)[r]{\strut{}$10^{-4}$}}%
      \put(860,2764){\makebox(0,0)[r]{\strut{}$10^{-3}$}}%
      \put(860,3188){\makebox(0,0)[r]{\strut{}$10^{-2}$}}%
      \put(860,3613){\makebox(0,0)[r]{\strut{}$10^{-1}$}}%
      \put(1413,440){\makebox(0,0){\strut{}$10^{-2}$}}%
      \put(2850,440){\makebox(0,0){\strut{}$10^{-1}$}}%
      \put(4287,440){\makebox(0,0){\strut{}$10^{0}$}}%
      \put(160,2126){\rotatebox{-270}{\makebox(0,0){\strut{}$\eta_R/|u|_{1,\Omega}$}}}%
      \put(2633,140){\makebox(0,0){\strut{}$h$}}%
      \put(1576,2917){\makebox(0,0)[l]{\strut{}\footnotesize$1$}}%
      \put(1258,2872){\makebox(0,0)[l]{\strut{}\footnotesize$1$}}%
      \put(1576,2298){\makebox(0,0)[l]{\strut{}\footnotesize$1$}}%
      \put(1258,2170){\makebox(0,0)[l]{\strut{}\footnotesize$2$}}%
      \put(2176,917){\makebox(0,0)[l]{\strut{}\footnotesize$1$}}%
      \put(2531,1267){\makebox(0,0)[l]{\strut{}\footnotesize$3$}}%
    }%
    \gplgaddtomacro\gplfronttext{%
      \csname LTb\endcsname%
      \put(3384,1203){\makebox(0,0)[r]{\strut{}$k=1$}}%
      \csname LTb\endcsname%
      \put(3384,1003){\makebox(0,0)[r]{\strut{}$k=2$}}%
      \csname LTb\endcsname%
      \put(3384,803){\makebox(0,0)[r]{\strut{}$k=3$}}%
    }%
    \gplbacktext
    \put(0,0){\includegraphics{SinSin_uniform_Eta}}%
    \gplfronttext
  \end{picture}%
\endgroup

%% file: fig/SinSin_uniform_L2.pdf_t
% GNUPLOT: LaTeX picture with Postscript
\begingroup
  \makeatletter
  \providecommand\color[2][]{%
    \GenericError{(gnuplot) \space\space\space\@spaces}{%
      Package color not loaded in conjunction with
      terminal option `colourtext'%
    }{See the gnuplot documentation for explanation.%
    }{Either use 'blacktext' in gnuplot or load the package
      color.sty in LaTeX.}%
    \renewcommand\color[2][]{}%
  }%
  \providecommand\includegraphics[2][]{%
    \GenericError{(gnuplot) \space\space\space\@spaces}{%
      Package graphicx or graphics not loaded%
    }{See the gnuplot documentation for explanation.%
    }{The gnuplot epslatex terminal needs graphicx.sty or graphics.sty.}%
    \renewcommand\includegraphics[2][]{}%
  }%
  \providecommand\rotatebox[2]{#2}%
  \@ifundefined{ifGPcolor}{%
    \newif\ifGPcolor
    \GPcolortrue
  }{}%
  \@ifundefined{ifGPblacktext}{%
    \newif\ifGPblacktext
    \GPblacktexttrue
  }{}%
  % define a \g@addto@macro without @ in the name:
  \let\gplgaddtomacro\g@addto@macro
  % define empty templates for all commands taking text:
  \gdef\gplbacktext{}%
  \gdef\gplfronttext{}%
  \makeatother
  \ifGPblacktext
    % no textcolor at all
    \def\colorrgb#1{}%
    \def\colorgray#1{}%
  \else
    % gray or color?
    \ifGPcolor
      \def\colorrgb#1{\color[rgb]{#1}}%
      \def\colorgray#1{\color[gray]{#1}}%
      \expandafter\def\csname LTw\endcsname{\color{white}}%
      \expandafter\def\csname LTb\endcsname{\color{black}}%
      \expandafter\def\csname LTa\endcsname{\color{black}}%
      \expandafter\def\csname LT0\endcsname{\color[rgb]{1,0,0}}%
      \expandafter\def\csname LT1\endcsname{\color[rgb]{0,1,0}}%
      \expandafter\def\csname LT2\endcsname{\color[rgb]{0,0,1}}%
      \expandafter\def\csname LT3\endcsname{\color[rgb]{1,0,1}}%
      \expandafter\def\csname LT4\endcsname{\color[rgb]{0,1,1}}%
      \expandafter\def\csname LT5\endcsname{\color[rgb]{1,1,0}}%
      \expandafter\def\csname LT6\endcsname{\color[rgb]{0,0,0}}%
      \expandafter\def\csname LT7\endcsname{\color[rgb]{1,0.3,0}}%
      \expandafter\def\csname LT8\endcsname{\color[rgb]{0.5,0.5,0.5}}%
    \else
      % gray
      \def\colorrgb#1{\color{black}}%
      \def\colorgray#1{\color[gray]{#1}}%
      \expandafter\def\csname LTw\endcsname{\color{white}}%
      \expandafter\def\csname LTb\endcsname{\color{black}}%
      \expandafter\def\csname LTa\endcsname{\color{black}}%
      \expandafter\def\csname LT0\endcsname{\color{black}}%
      \expandafter\def\csname LT1\endcsname{\color{black}}%
      \expandafter\def\csname LT2\endcsname{\color{black}}%
      \expandafter\def\csname LT3\endcsname{\color{black}}%
      \expandafter\def\csname LT4\endcsname{\color{black}}%
      \expandafter\def\csname LT5\endcsname{\color{black}}%
      \expandafter\def\csname LT6\endcsname{\color{black}}%
      \expandafter\def\csname LT7\endcsname{\color{black}}%
      \expandafter\def\csname LT8\endcsname{\color{black}}%
    \fi
  \fi
  \setlength{\unitlength}{0.0500bp}%
  \begin{picture}(4648.00,3854.00)%
    \gplgaddtomacro\gplbacktext{%
      \csname LTb\endcsname%
      \put(980,640){\makebox(0,0)[r]{\strut{}$10^{-11}$}}%
      \put(980,970){\makebox(0,0)[r]{\strut{}$10^{-10}$}}%
      \put(980,1301){\makebox(0,0)[r]{\strut{}$10^{-9}$}}%
      \put(980,1631){\makebox(0,0)[r]{\strut{}$10^{-8}$}}%
      \put(980,1961){\makebox(0,0)[r]{\strut{}$10^{-7}$}}%
      \put(980,2292){\makebox(0,0)[r]{\strut{}$10^{-6}$}}%
      \put(980,2622){\makebox(0,0)[r]{\strut{}$10^{-5}$}}%
      \put(980,2952){\makebox(0,0)[r]{\strut{}$10^{-4}$}}%
      \put(980,3283){\makebox(0,0)[r]{\strut{}$10^{-3}$}}%
      \put(980,3613){\makebox(0,0)[r]{\strut{}$10^{-2}$}}%
      \put(1517,440){\makebox(0,0){\strut{}$10^{-2}$}}%
      \put(2902,440){\makebox(0,0){\strut{}$10^{-1}$}}%
      \put(4287,440){\makebox(0,0){\strut{}$10^{0}$}}%
      \put(160,2126){\rotatebox{-270}{\makebox(0,0){\strut{}$\|u-u_h\|_{0,\Omega}/\|u\|_{0,\Omega}$}}}%
      \put(2693,140){\makebox(0,0){\strut{}$h$}}%
      \put(1675,2544){\makebox(0,0)[l]{\strut{}\footnotesize$1$}}%
      \put(1367,2434){\makebox(0,0)[l]{\strut{}\footnotesize$2$}}%
      \put(1675,1961){\makebox(0,0)[l]{\strut{}\footnotesize$1$}}%
      \put(1367,1762){\makebox(0,0)[l]{\strut{}\footnotesize$3$}}%
      \put(2351,820){\makebox(0,0)[l]{\strut{}\footnotesize$1$}}%
      \put(2643,1201){\makebox(0,0)[l]{\strut{}\footnotesize$4$}}%
    }%
    \gplgaddtomacro\gplfronttext{%
      \csname LTb\endcsname%
      \put(3384,1203){\makebox(0,0)[r]{\strut{}$k=1\!\!$}}%
      \csname LTb\endcsname%
      \put(3384,1003){\makebox(0,0)[r]{\strut{}$k=2\!\!$}}%
      \csname LTb\endcsname%
      \put(3384,803){\makebox(0,0)[r]{\strut{}$k=3\!\!$}}%
    }%
    \gplbacktext
    \put(0,0){\includegraphics{SinSin_uniform_L2}}%
    \gplfronttext
  \end{picture}%
\endgroup

%% file: fig/PsiK2_adaptive.pdf_t
% GNUPLOT: LaTeX picture with Postscript
\begingroup
  \makeatletter
  \providecommand\color[2][]{%
    \GenericError{(gnuplot) \space\space\space\@spaces}{%
      Package color not loaded in conjunction with
      terminal option `colourtext'%
    }{See the gnuplot documentation for explanation.%
    }{Either use 'blacktext' in gnuplot or load the package
      color.sty in LaTeX.}%
    \renewcommand\color[2][]{}%
  }%
  \providecommand\includegraphics[2][]{%
    \GenericError{(gnuplot) \space\space\space\@spaces}{%
      Package graphicx or graphics not loaded%
    }{See the gnuplot documentation for explanation.%
    }{The gnuplot epslatex terminal needs graphicx.sty or graphics.sty.}%
    \renewcommand\includegraphics[2][]{}%
  }%
  \providecommand\rotatebox[2]{#2}%
  \@ifundefined{ifGPcolor}{%
    \newif\ifGPcolor
    \GPcolortrue
  }{}%
  \@ifundefined{ifGPblacktext}{%
    \newif\ifGPblacktext
    \GPblacktexttrue
  }{}%
  % define a \g@addto@macro without @ in the name:
  \let\gplgaddtomacro\g@addto@macro
  % define empty templates for all commands taking text:
  \gdef\gplbacktext{}%
  \gdef\gplfronttext{}%
  \makeatother
  \ifGPblacktext
    % no textcolor at all
    \def\colorrgb#1{}%
    \def\colorgray#1{}%
  \else
    % gray or color?
    \ifGPcolor
      \def\colorrgb#1{\color[rgb]{#1}}%
      \def\colorgray#1{\color[gray]{#1}}%
      \expandafter\def\csname LTw\endcsname{\color{white}}%
      \expandafter\def\csname LTb\endcsname{\color{black}}%
      \expandafter\def\csname LTa\endcsname{\color{black}}%
      \expandafter\def\csname LT0\endcsname{\color[rgb]{1,0,0}}%
      \expandafter\def\csname LT1\endcsname{\color[rgb]{0,1,0}}%
      \expandafter\def\csname LT2\endcsname{\color[rgb]{0,0,1}}%
      \expandafter\def\csname LT3\endcsname{\color[rgb]{1,0,1}}%
      \expandafter\def\csname LT4\endcsname{\color[rgb]{0,1,1}}%
      \expandafter\def\csname LT5\endcsname{\color[rgb]{1,1,0}}%
      \expandafter\def\csname LT6\endcsname{\color[rgb]{0,0,0}}%
      \expandafter\def\csname LT7\endcsname{\color[rgb]{1,0.3,0}}%
      \expandafter\def\csname LT8\endcsname{\color[rgb]{0.5,0.5,0.5}}%
    \else
      % gray
      \def\colorrgb#1{\color{black}}%
      \def\colorgray#1{\color[gray]{#1}}%
      \expandafter\def\csname LTw\endcsname{\color{white}}%
      \expandafter\def\csname LTb\endcsname{\color{black}}%
      \expandafter\def\csname LTa\endcsname{\color{black}}%
      \expandafter\def\csname LT0\endcsname{\color{black}}%
      \expandafter\def\csname LT1\endcsname{\color{black}}%
      \expandafter\def\csname LT2\endcsname{\color{black}}%
      \expandafter\def\csname LT3\endcsname{\color{black}}%
      \expandafter\def\csname LT4\endcsname{\color{black}}%
      \expandafter\def\csname LT5\endcsname{\color{black}}%
      \expandafter\def\csname LT6\endcsname{\color{black}}%
      \expandafter\def\csname LT7\endcsname{\color{black}}%
      \expandafter\def\csname LT8\endcsname{\color{black}}%
    \fi
  \fi
  \setlength{\unitlength}{0.0500bp}%
  \begin{picture}(7370.00,3968.00)%
    \gplgaddtomacro\gplbacktext{%
      \csname LTb\endcsname%
      \put(860,640){\makebox(0,0)[r]{\strut{}$10^{-3}$}}%
      \put(860,1257){\makebox(0,0)[r]{\strut{}$10^{-2}$}}%
      \put(860,1875){\makebox(0,0)[r]{\strut{}$10^{-1}$}}%
      \put(860,2492){\makebox(0,0)[r]{\strut{}$10^{0}$}}%
      \put(860,3110){\makebox(0,0)[r]{\strut{}$10^{1}$}}%
      \put(860,3727){\makebox(0,0)[r]{\strut{}$10^{2}$}}%
      \put(980,440){\makebox(0,0){\strut{}$10^{1}$}}%
      \put(2327,440){\makebox(0,0){\strut{}$10^{2}$}}%
      \put(3673,440){\makebox(0,0){\strut{}$10^{3}$}}%
      \put(5020,440){\makebox(0,0){\strut{}$10^{4}$}}%
      \put(6366,440){\makebox(0,0){\strut{}$10^{5}$}}%
      \put(160,2183){\rotatebox{-270}{\makebox(0,0){\strut{}$\|u-u_h\|_b$, $\eta_R$}}}%
      \put(3994,140){\makebox(0,0){\strut{}number of degrees of freedom}}%
      \put(6105,2601){\makebox(0,0)[l]{\strut{}\footnotesize$2$}}%
      \put(6694,2563){\makebox(0,0)[l]{\strut{}\footnotesize$1$}}%
      \put(5173,1945){\makebox(0,0)[l]{\strut{}\footnotesize$1$}}%
      \put(5364,1875){\makebox(0,0)[l]{\strut{}\footnotesize$1$}}%
      \put(4335,832){\makebox(0,0)[l]{\strut{}\footnotesize$2$}}%
      \put(4079,952){\makebox(0,0)[l]{\strut{}\footnotesize$3$}}%
    }%
    \gplgaddtomacro\gplfronttext{%
      \csname LTb\endcsname%
      \put(2540,1803){\makebox(0,0)[r]{\strut{}$\|\cdot\|_b,\,k=1$}}%
      \csname LTb\endcsname%
      \put(2540,1603){\makebox(0,0)[r]{\strut{}$\|\cdot\|_b,\,k=2$}}%
      \csname LTb\endcsname%
      \put(2540,1403){\makebox(0,0)[r]{\strut{}$\|\cdot\|_b,\,k=3$}}%
      \csname LTb\endcsname%
      \put(2540,1203){\makebox(0,0)[r]{\strut{}$\eta_R,\,k=1$}}%
      \csname LTb\endcsname%
      \put(2540,1003){\makebox(0,0)[r]{\strut{}$\eta_R,\,k=2$}}%
      \csname LTb\endcsname%
      \put(2540,803){\makebox(0,0)[r]{\strut{}$\eta_R,\,k=3$}}%
    }%
    \gplbacktext
    \put(0,0){\includegraphics{PsiK2_adaptive}}%
    \gplfronttext
  \end{picture}%
\endgroup

%% file: Conclusion.tex
\section{Conclusion}
This publication contains one of the first results on a posteriori error control for conforming approximation methods on polygonal meshes. Such general meshes are very flexible and convenient especially in adaptive mesh refinement strategies. But, their potential application areas are not yet fully exploited. The presented results are a building block for future developments of efficient numerical methods involving polygonal discretizations.